\documentclass[11pt,oneside,english,reqno]{amsart}
\usepackage[T1]{fontenc}
\usepackage[latin9,utf8]{inputenc}
\usepackage[a4paper]{geometry}
\usepackage{mathrsfs}
\usepackage{amstext}
\usepackage{amsthm}
\usepackage{amssymb}
\usepackage{color}
\usepackage{hyperref}
\usepackage{enumerate}
\usepackage{verbatim} 
\usepackage{stmaryrd}
\usepackage{enumitem}
\usepackage{bbm}

\makeatletter
\numberwithin{equation}{section}
\numberwithin{figure}{section}
\theoremstyle{plain}
\newtheorem{thm}{\protect\theoremname}
\theoremstyle{definition}
\newtheorem{defn}[thm]{\protect\definitionname}
\theoremstyle{remark}
\newtheorem{rem}[thm]{\protect\remarkname}
\theoremstyle{plain}
\newtheorem{lem}[thm]{\protect\lemmaname}
\theoremstyle{plain}
\newtheorem{prop}[thm]{\protect\propositionname}
\theoremstyle{plain}
\newtheorem{cor}[thm]{\protect\corollaryname}
\theoremstyle{plain}
\newtheorem{hyp}[thm]{\protect\hypothesisname}

\makeatother

\usepackage{babel}
\providecommand{\corollaryname}{Corollary}
\providecommand{\definitionname}{Definition}
\providecommand{\lemmaname}{Lemma}
\providecommand{\propositionname}{Proposition}
\providecommand{\remarkname}{Remark}
\providecommand{\theoremname}{Theorem}
\providecommand{\hypothesisname}{Hypothesis}

\newcommand{\cA}{\mathcal{A}}
\newcommand{\cB}{\mathcal{B}}
\newcommand{\cC}{\mathcal{C}}

\newcommand{\cF}{\mathcal{F}}
\newcommand{\cG}{\mathcal{G}}

\newcommand{\cI}{\mathcal{I}}

\newcommand{\cL}{\mathcal{L}}

\newcommand{\cN}{\mathcal{N}}

\newcommand{\cP}{\mathcal{P}}

\newcommand{\cS}{\mathcal{S}}

\newcommand{\cX}{\mathcal{X}}

\newcommand{\EE}{\mathbb{E}}

\newcommand{\NN}{\mathbb{N}}

\newcommand{\PP}{\mathbb{P}}

\newcommand{\RR}{\mathbb{R}}

\newcommand{\TT}{\mathbb{T}}

\newcommand{\la}{\langle}
\newcommand{\ra}{\rangle}

\newcommand{\scC}{\mathscr{C}}
\newcommand{\vertiii}[1]{{\left\vert\kern-0.25ex\left\vert\kern-0.25ex\left\vert #1 
    \right\vert\kern-0.25ex\right\vert\kern-0.25ex\right\vert}}

\newcommand{\eps}{\varepsilon}
\newcommand{\varsig}{\varsigma}

\newcommand{\dd}{\mathop{}\!\mathrm{d}}

\newcommand{\llrr}{\langle \cdot\rangle}
 
 \hypersetup{
     colorlinks   = true,
     citecolor    = blue,
     linkcolor    = blue
}

\begin{document}

\title[Regularization of mSHE through irregular perturbation]{Pathwise regularization of the stochastic heat equation with multiplicative noise through irregular  perturbation}


\author{R\'emi Catellier \and Fabian A. Harang}

\address{R\'emi Catellier:
 Universit\'e C\^ote d Azur, CNRS, LJAD, France}
\email{remi.catellier@univ-cotedazur.fr} 

\address{Fabian A. Harang:
 Department of Mathematics, University of Oslo, P.O. box 1053, Blindern, 0316, OSLO, Norway}
\email{fabianah@math.uio.no} 

\thanks{\emph{Acknowledgments.} 
FH gratefully acknowledges financial support from the STORM project 274410, funded by the Research Council of Norway. We would also thank Nicolas Perkowski for several fruitful  discussions on this topic.}

\begin{abstract}
   Existence and uniqueness of solutions to the stochastic heat equation with multiplicative  spatial noise is studied. In the spirit of pathwise regularization by noise, we show that a perturbation by a sufficiently irregular continuous path establish wellposedness of such equations, even when the drift and diffusion coefficients are given as generalized functions or distributions. In addition we prove regularity of the averaged field associated to a L\'evy fractional stable motion, and use this as an example of a perturbation regularizing the multiplicative stochastic heat equation. 
\end{abstract}

\maketitle 

{
\hypersetup{linkcolor=black}
\tableofcontents
}

\section{Introduction}
The stochastic heat equation with multiplicative noise (mSHE), is given on the form
\begin{equation}\label{eq:full SHE}
    \partial_t u = \Delta u+ b(u)+g(u)\xi, \qquad  u_0\in \cC^\beta,
\end{equation}
where $\xi$ is a space time noise on  $\RR^d$ or $\TT^d$ and $\cC^\beta$ is the Besov H\"older space of $\RR^d$ or $\TT^d$, and $b$ and $g$ are sufficiently smooth functions. 
This equation is a fundamental  stochastic partial differential equation, and is applied for modelling in a diverse selection of natural sciences, ranging from chemistry and biology to physics. 
The existence and uniqueness of \eqref{eq:full SHE} is typically proven under the condition that both $b$ and $g$ are Lipschitz functions of linear growth (see e.g. \cite{HUNUALARTSONG2013}). 
When taking a pathwise approach to a solution theory, even more regularity of these non linear functions may be required (see for instance \cite{Bailleul2019} where $g\in C^3$ is required and Appendix \ref{sec:CL-standard} for a proof a pathwise wellposedness in a simple context). Of course if $g\equiv 0$, then \eqref{eq:full SHE} is known as the (deterministic) non-linear heat equation, for which  uniqueness fails in general under weaker conditions on $b$ than Lipschitz and linear growth. 

Motivated by this, a natural question to ask is if it is possible to prove existence and uniqueness of \eqref{eq:full SHE w t noise} under weaker conditions on $b$ and $g$. 
Inspired by the well known regularization by noise phenomena in stochastic differential equations, one may think that the same principles of regularization would extend to the case of stochastic partial differential equations like \eqref{eq:full SHE w t noise}. In this article,
we aim at giving some insights into this question by investigating  \eqref{eq:full SHE w t noise} in a fully pathwise manner under perturbation by a measurable time (only) dependent path.  In fact, we will prove that a perturbation by a sufficiently irregular path yields wellposedness of the mSHE, even for distributional (generalized functions)  coefficients $g$ and $b$.  In the next section we give a more detailed description of the specific equation under consideration, and the techniques that we apply in order to prove this regularizing effect of the perturbation.

\subsection{Methodology}
 Inspired by the theory of regularization by noise for ordinary or stochastic differential equations, we show that a suitably chosen measurable path $w:[0,T]\rightarrow \RR^d$ provides a regularizing effect on \eqref{eq:full SHE}, by considering the formal equation
\begin{equation}\label{eq:full SHE w t noise}
    \partial_t u=\Delta u+b(u)+g(u)\xi+\dot{\omega}_t,\qquad u_0\in \cC^\beta,
\end{equation}
where $\xi$ is a spatial (distributional) noise taking values in $\RR^d$, and $\dot{\omega}_t$ is the distributional derivative of a continuous path $\omega$. 
To this end, we formulate \eqref{eq:full SHE w t noise} in terms of the non-linear Young framework, developed in \cite{Catellier2016,galeati2020noiseless,galeati2020nonlinear}. We extend this framework to the infinite dimensional setting  adapted to Volterra type integrals appearing when considering the mild formulation of \eqref{eq:full SHE w t noise}. The integration framework developed here is strongly based on the recently developed {\em Volterra sewing lemma} of \cite{harang2019tindel}, and  does not require any semi-group property of the Volterra operator. The integral can therefore be applied to several different  problems relating to infinite dimensional Volterra integration, and thus  we believe that this construction is interesting in itself.

To motivate the methodology of the current paper, consider again \eqref{eq:full SHE w t noise}  and  set $\theta=u-w$ with $\theta_0=u_0\in \cC^\beta$ then formally $\theta$ solves the following integral equation
\begin{equation*}
    \theta_t=P_{t}\theta_0+\int_0^t P_{t-s}  b(\theta_s+\omega_s)\dd s+\int_0^t P_{t-s} \xi g(\theta_s+\omega_s)\dd s,
\end{equation*}
where $P$ is the fundamental solution operator associated with the heat equation, and the product  $P_t \theta_0$ is interpreted as spatial convolution. For simplicity, we will carry out most of our analysis with $b\equiv 0$, as this term is indeed easier to handle than the term with multiplicative noise. The equation we will then consider is given by 
\begin{equation}\label{eq:first theta eq}
    \theta_t=P_{t}\theta_0+\int_0^t P_{t-s} \xi g(\theta_s+\omega_s)\dd s.
\end{equation}
In Section \ref{sec:drifted mshe} we provide a detailed description of how our results can easily be extended to include the drift term with distributional $b$, by simply appealing to the the analysis carried out for the purely multiplicative equation \eqref{eq:first theta eq}. 
Associated to the path $\omega$ and the distribution $g$, define the averaged distribution $T^\omega g:[0,T]\times \RR^d \rightarrow \RR^d $ by the mapping 
\begin{equation*}
    (t,x)\mapsto \int_0^t g (x+\omega_s)\dd s.
\end{equation*}
After proving that for certain paths $\omega$ the distribution $T^\omega g$ is in fact a regular function,  we then consider \eqref{eq:first theta eq}  as a non-linear Young equation of the form
\begin{equation}\label{eq:NLY intro}
    \theta_t = P_t\theta_0+\int_0^t P_{t-s}\xi T^\omega_{\dd s}g(\theta_s), \quad \theta_0\in \cC^\beta.
\end{equation}
Here, the integral is interpreted in an infinite dimensional non-linear Young-Volterra sense. That is, to suit our purpose, we extend the non-linear Young integral to an infinite dimensional setting as well as allowing for the action of a Volterra operator on the integrand. The integral is then constructed as the Banach  valued element
\begin{equation*}
    \int_0^t P_{t-s}T^\omega_{\dd s}g(\theta_s):=\lim_{|\cP|\rightarrow 0} \sum_{[u,v]\in \cP } P_{t-u}\xi T^\omega_{u,v} g(\theta_u),
\end{equation*}
where $T^\omega_{u,v} g:=T^\omega_{v} g-T^\omega_{u} g$. 
We stress that in contrast to the non-linear Young integral used for example in  \cite{Catellier2016,galeati2020noiseless,harang2020cinfinity,galeati2020nonlinear},  the above integral is truly an infinite dimensional object, and extra care must be taken when building it from the averaged function $T^\omega g$. 
Indeed, for each $t\geq0$, $\theta_t\in \cC^\beta(\RR^d;\RR)$ and so the function $T^\omega g$ is then lifted to be a functional on $\cC^\beta$. We show that this lift comes at the cost of an extra degree of assumed regularity on the averaged function $T^\omega g$. 
Furthermore, due to the assumption that $\xi\in \cC^{-\vartheta} $ for $\vartheta>0$ (i.e. $\xi$ is assumed to be truly distributional) we need to make use of the product in Besov space in order to make the product of $\xi T^\omega_{u,v} g(\theta_u) $ well defined.

Similar to the theory of rough paths, our analysis can be divided into two parts: (i) a probabilistic step, and (ii) a  deterministic (analytic) step. We give a short description of the two steps here:  

\begin{itemize}[leftmargin=0.7cm]
    \item[(i)] Let $E$ be a separable Banach space. We  develop an abstract framework of existence and uniqueness of Banach valued equations 
\begin{equation}\label{eq:abs intro}
    \theta_t =p_t+\int_0^t S_{t-s}X_{\dd s}(\theta_s), 
\end{equation}
where $S$ is a suitable (possibly singular) Volterra operator, and $X:[0,T]\times E\rightarrow E$ is a function which is $\frac{1}{2}+$ H\"older regular in time, and suitably regular in its spatial argument (to be specified later), and $p:[0,T]\rightarrow  E$ is a sufficiently regular function.  To this end, we use a simple extension of the Volterra sewing lemma developed in \cite{harang2019tindel} to  construct the non-linear Young-Volterra integral appearing in \eqref{eq:abs intro} as the following 
\begin{equation*}
    \int_0^t S_{t-s}X_{\dd s}(\theta_s):=\lim_{|\cP|\rightarrow 0}\sum_{[u,v]\in \cP[0,t]} S_{t-u}X_{u,v}(\theta_u),
\end{equation*}
where $\cP[0,t]$ is a partition of $[0,t]$ with mesh size $|\cP|$ converging to zero.
\\

\item[(ii)] The second step is then to consider $\{\omega_t\}_{t\in [0,T]}$ to be a stochastic process on a probability space $(\Omega,\cF,\PP)$, and  we need to show that the averaged function $T^\omega g$ is indeed a sufficiently regular function $\PP$-a.s., even when $g$ is a true distribution.
This is done by probabilistic methods. 
\end{itemize}

At last we relate the abstract function $X$ from \eqref{eq: abs equation} to the averaged function $T^\omega$, and then a combination of the previous two steps gives us existence and uniqueness of \eqref{eq:NLY intro}, and thus also \eqref{eq:first theta eq}, which then is used to makes sense of \eqref{eq:full SHE w t noise} through the translation $u=v+\omega$.

\subsection{Short overview of existing literature}
In order to prove existence and uniqueness of \eqref{eq:full SHE}, there is two main directions to follow: the classical probabilistic setting based on It\^o type theory, or the pathwise approach based on rough paths or similar techniques. Using the first approach,   one typically require that $b$ and $g$ are Lipschitz and of linear growth (see e.g. \cite{HUNUALARTSONG2013}). Note that in this context, at least in dimension $d=1$, one can prove also prove existence and uniqueness with less restrictive requirements on $g$ in certain cases (see \cite{muellerNonuniquenessParabolicSPDE2014, neumanPathwiseUniquenessStochastic2018, yangPathwiseUniquenessSPDE2017}). In particular in \cite{muellerNonuniquenessParabolicSPDE2014}, the authors prove, using probabilistic arguments, that there is a non-zero solution to the equation 
\[\partial_t u = \Delta u + |u|^\gamma \xi + \psi,\quad u(0,\cdot) = 0,\]
where $\xi$ is a space time white noise in dimension $1+1$, $\psi$ is a non-zero, non-negative function smooth compactly supported function and $0\leq \gamma<\frac{3}{4}$. They also prove that when $\gamma>\frac34$ uniqueness holds. 
Note also that when $\xi$ is a deterministic function and not a distribution, wellposedness is a well-know topic. In particular if $\xi$ is a non-negative continuous and bounded function, Fujita and Watanabe \cite{fujitaUniquenessNonuniquenessSolutions1968} prove that Osgood condition on $g$ are "nearly" necessary and sufficient to guaranty uniqueness. In particular when $g$ is only H\"older continuous, one can not expect to have uniqueness. One can also consult \cite{benediktNonuniquenessSolutionsInitialvalue2015} and the reference therein for further attempt in that direction.

When using pathwise techniques to solve (stochastic) nonlinear heat equation, as usual (see \cite{Davie2007} for counterexamples in a rough path context), one typically needs to require even higher regularity on $b$ and $g$ in order guarantee existence and uniqueness (see e.g. \cite{Bailleul2019} where three times differentiability is assumed and the Appendix \ref{sec:CL-standard} for a simple proof). 

To the best of our knowledge, little has been done in the direction of investigating the regularizing effects obtained from  measurable perturbations of the heat equation. In the case when $g\equiv0$ it has been proven in \cite{Nualart2004Ouknine} (see also \cite{butkovsky2019}) that the additive stochastic heat equation on the form 
\begin{equation*}
    \partial_t u=\partial^2_xu + b(u)+\partial_t\partial_{x} \omega,\quad (t,x)\in [0,T]\times [0,1]
\end{equation*}
exists uniquely, even when $b$ is only bounded and measurable and $\partial_t\partial_{x} \omega$ is understood as a white noise on  $[0,T]\times [0,1]$. Thus the addition of noise seems to give similar regularizing effects in the stochastic heat equation as is observed in SDEs. 

Note that the recent publication  \cite{athreyaWellposednessStochasticHeat2020} continue in the investigation when $g\equiv 1$, and in particular the authors are able to recover results on skew stochastic heat equation.

\subsection{Main results}

Before presenting our main results, let us first give a definition of what we will call a solution to \eqref{eq:full SHE w t noise}. Let us remind that the definition of admissible weight and Besov spaces is written in Appendix \ref{sub:weighted_besov} and \ref{sub:weighted_lebesgue}, the precise definition of non-linear Young-Volterra equation is given in Section \ref{sec:sewing} and the definition of averaged field is given in Section \ref{averaged fields}. The equations considered in the below results can either be considerd on $\RR$ or $\TT$. Indeed, all the following results are true in both settings. Furthermore, for the sake of using the space white noise, it might be instructive to think to the space as the torus $\TT$. 

\begin{defn}\label{def: concept of solution}
Let $\omega:[0,T]\rightarrow \RR$ be a measurable path, and consider a $g\in \cS'$ such that the averaged field  $T^\omega b\in \cC^\gamma_T \cC^\kappa(w)$ for some $\gamma>\frac{1}{2}$ and  $\kappa\geq 3$ and an admissible weight $w:\RR^d\rightarrow \RR$. 
Suppose that $0 < \vartheta < \beta < 1$ and suppose that $\xi$ is a spatial noise contained in $\cC^{-\vartheta}$. 
Take $\rho = \frac{\beta + \vartheta}{2}$ and suppose that $\gamma-\rho>1-\gamma$. Let $ u_0 \in \cC^\beta$. 
We say that there exists a unique solution to the equation 
\begin{equation}\label{eq:formalequationdef}
    u_t=P_tu_0+\int_0^t P_{t-s}\xi g(u_s)\dd s+\omega_t,\quad t\in [0,T]
\end{equation}
if $u\in \omega + \scC^\varsig_T\cC^\beta$ for any $1-\gamma<\varsig<\gamma-\rho$, and there exists a unique $\theta\in \scC^\varsig_T\cC^\beta$ such that  $u=\omega+\theta$ and $\theta$ solves the non-linear Young equation
\begin{equation}\label{eq:thetaeq}
    \theta_t =P_{t}u_0 +\int_0^t P_{t-s}\xi T^\omega_{\dd s}g(\theta_s).
\end{equation}
Here the integral is understood as a non-linear Young--Volterra integral, as constructed in Section \ref{sec:sewing} and the space $\scC^\varsig_T$ is defined in Definition \ref{def:explosive-norm}. 
\end{defn}

\begin{thm}\label{thm:main existence and uniquness with existing avg op}
Let $\omega:[0,T]\rightarrow \RR$ be a measurable path, and consider a distribution  $g\in \cS'(\RR)$ such that the associated  averaged field  $T^\omega b\in \cC^\gamma_T \cC^\kappa(w)$ for some $\gamma>\frac{1}{2}$,  $\kappa\geq 3$ and a admissible weight function $w:\RR^d \rightarrow \RR_+$. %
Suppose that $0<\vartheta< 1$ and take $\vartheta<\beta < 2- \vartheta$ and that $\xi$ is a spatial noise contained in $\cC^{-\vartheta}$. Let $\rho = \frac{\beta + \vartheta}{2}$ and assume that $1-\gamma < \gamma - \rho$.
Then there exists a time $\tau\in (0,T]$ such that there exists  a unique solution $u$ in the sense of Definition \ref{def: concept of solution} to Equation \eqref{eq:formalequationdef}. 
If $w$ is globally bounded, then the solution is global, in the sense that a unique solution $u$ to \eqref{eq:formalequationdef}  is define on all $[0,\tau]$ for any $\tau\in (0,T]$. 
\end{thm}

The unique solution can be interpreted to be a "physical" one, in the sense that it is stable under approximations. We summarize this in the following corollary. 

\begin{cor}\label{cor:main stability of solutions}
Under the assumptions of Theorem \ref{thm:main existence and uniquness with existing avg op}, if $\{g_n\}_{n\in \NN}$ is a sequence of smooth functions converging to $g\in \cS'(\RR)$ such that $T^\omega g_n\rightarrow T^\omega g\in \cC^\gamma_T\cC^\kappa(w)$, then the corresponding sequence of solutions $\{u_n\}_{n\in\NN}=\{\omega + \theta^n\}_{n\in \NN}$ to equation \eqref{eq:thetaeq} converge to $u = w+\theta$ in the sense that $\theta^n \to \theta $ in $ \scC^\varsig_\tau\cC^\beta$ for any $1-\gamma < \varsig < \gamma-\rho$. 
\end{cor}

In applications, one is typically interested in the regularizing effects provided by specific sample paths of stochastic processes $\omega$. Although the class of regularizing paths is already well developed, we show here as an example the regularizing effect of measurable sample paths of fractional L\'evy processes (see Section \ref{sec: Averaging operators with Levy noise}). 

In this connection suppose that $\xi$ is a spatial white noise on the torus $\TT$, and we find a conditions on the distribution $g$ and the initial data $u_0$ so that a unique solution to \eqref{eq:formalequationdef} exists.

\begin{thm}\label{thm:main fractional Levy}
Let $\alpha \in (0,2]$. Let $H \in (0,1) \cap \{\alpha^{-1}\}$. Let $L^H$ be a Linear Fractional L\'evy Process with Hurst parameter $H$ built from a symmetric $\alpha$-stable L\'evy process as defined in Section \ref{sec: Averaging operators with Levy noise}. 
Let $\vartheta \in (0,1)$ and let $\kappa > 3 - \frac{1-\vartheta}{2H}$. 
There exists $\eps>0$ small enough such that for all $\xi \in \cC^{-\vartheta}$ and all $g\in \cC^{\kappa}(w)$ where $w$ is an admissible weight, almost surely for all $u_0\in\cC^{\vartheta + \eps}$ there exists a unique local solution to the mSHE in the sense of Definition \ref{def: concept of solution}. If $w$ is bounded, then the solution exists globally. 
\end{thm}

One can specify the previous theorem by letting $\xi$ to be the space white noise :

\begin{cor}\label{cor:white noise}
Let $d= 1$ and let $\xi$ be a space white noise on the Torus $\TT$. Let $\alpha \in (0,2]$, let $H\in(0,1)\cap\{\alpha^{-1}\}$ and let $\kappa > 3 - \frac{1}{4H}$. Let $w$ be an admissible weight. There exists $\eps>0$ small enough such that for all $g\in \cC^{\kappa}(w)$, almost surely for all  
$u_0\in\cC^{\frac{1}{2}+\eps}$ and let there exists a unique solution (in the sense of Definition \ref{def: concept of solution}) to the mSHE
\[\dd u  = \Delta u \dd t + g(u) \xi\dd t + \dd L^H_t,\quad u(0,\cdot) = u_0, \]
where $L^H$ is a linear fractional stable motion.

In particular taking $H< \frac14$ allows us to take $\kappa<2$ and go beyond the classical theory using Bony estimates for the product of distribution in Besov spaces. When $H<\frac{1}{8}$ one can deal with non Lipschitz-continuous $g$ and when $H<\frac1{12}$, one can deal with distributional field $g$.
\end{cor}

The proofs of the above theorems and corollary can be found in Sections \ref{sec: existence and uniqueness of mSHE} and \ref{sec: Averaging operators with Levy noise}. 

\subsection{Outline of the paper}
The paper is structured as follows: In section \ref{sec:sewing} we extend the concept of non--linear Young integration to the infinite dimensional setting, including a Volterra operator, which in later sections will play the role of the inverse Laplacian. We also give a result on existence and uniqueness of abstract equations in Banach spaces. 

In section \ref{averaged fields} we give a short overview on the concept of averaged fields and their properties. As this topic is by now well studied in the literature, we only give here the necessary details and provide several references for further information. We also show how the standard averaged field can be viewed as an operator on certain Besov function spaces.

In section \ref{sec: existence and uniqueness of mSHE} we formulate the multiplicative stochastic heat equation in the non-linear Young-Volterra integration framework, using the concept of averaged fields. We prove existence and uniqueness of these equations, as well as Theorem \ref{thm:main existence and uniquness with existing avg op} and Corollary \ref{cor:main stability of solutions}. In section \ref{sec: Averaging operators with Levy noise} we investigate closer the regularity of averaged fields associated with sample paths of fractional Levy processes, and prove Theorem \ref{thm:main fractional Levy}. 

At last, in Section \ref{sec:conclussion} we give a short reflection on the main results of the article and provide some thoughts on future extensions of our results.

For thee sake of self-containedness, 
we have included an appendix with preliminaries on weighted Besov spaces, and some results regarding the regularity/singularity of the inverse Laplacian acting on Besov distributions.   In addition we give a short proof for existence and uniqueness of \eqref{eq:full SHE w t noise} in the case of twice differentiable $g$.

\subsection{Notation}
For $\beta \in \mathbb{R}$ and $d\geq 1$, we define the H\"older Besov space
\[\cC^{\beta} = B^\beta_{\infty,\infty}(\RR^d;\RR)\]
endowed with its usual norm built upon Paley-Littlewood blocks and denoted by $\|\cdot\|_{\cC^\beta}$ (see Appendix \ref{sub:weighted_besov} and especially Proposition \ref{prop:equivalent_norms} for more on weighted Besov spaces).  
Note that when working with a weight $w$ and weighted spaces, we denote the spaces $\cC^{\beta}(w)$ and the norm $\|\cdot\|_{\cC^{\beta}(w)}$.
For $T>0$ and $\varsig\in(0,1)$ and $E$ a Banach space with norm $\|\cdot\|_E$ and for $f:[0,T]\mapsto E$,
 \[ [ f]_{\varsig;E} = \sup_{s\neq t} \frac{\|f_t - f_s\|_E}{|t-s|^\varsig} < \infty,\]
 and for $\varsig>0$ we define 
 \[\|f\|_{\varsig;E} := \sum_{k=0}^{[\varsig ]} \|f^{([k])}_0\|_E + \big[ f^{([\varsig ])} \big]_{\varsig -[\varsig ];E},\]
 and finally
\[ \cC^{\varsig }_{T}E = \cC^{\varsig }\big([0,T]; E\big) := \left\{ f : [0,T] \to E\, :\, \|f\|_{\varsig;E} < \infty\right\}.\]
Whenever the underlying space $E$ is clear form the context we will use the short hand notation $[f]_\varsig$ etc. to denote the the H\"older semi norm. 
We will frequently use the increment notation $f_{s,t}:=f_t-f_s$. We write $f \lesssim g$ if there exists a constant $C>0$ such that $f\leq C g$. Furthermore to stress that the constant $C$ depends on a parameter $p$ we write $ l \lesssim_p g$. We write $f\simeq g$ is $f\lesssim g$ and $g \lesssim f$.
For $s\leq t$, we denote by  $\cP([s,t])$ a partition of the interval $[s,t]$.
We write $|\cP| = \max_{k\in\{0,\cdots,n-1\}}|t_{k+1}-t_k|$.

\section{Non-linear Young-Volterra theory in Banach spaces}

For $\alpha\in (0,2]$, let $t\mapsto P_t^{\frac{\alpha}{2}}$ be the fundamental solution associated to the $\alpha$-fractional heat equation
\[\partial_t P^{\frac{\alpha}{2}} = -(-\Delta)^{\frac{\alpha}{2}}P^{\frac\alpha2},\quad P^{\frac\alpha2}_0 = \delta_0.\] 
For a function $y:\RR^d \rightarrow \RR$ let $P^{\frac{\alpha}{2}}_t y$ denotes the convolution between the fundamental solution at time $t\geq 0$ and $y$. Note that for $\alpha=2$, we obtain the classical heat equation, and $P_\cdot:=P^1_\cdot $ then denotes the convolution operator with the Gaussian kernel.  Towards a pathwise analysis of stochastic parabolic equations, one encounter the problem that for a function $y\in \cC^{\kappa}(\RR^d)$ with  $\kappa\geq 0$, the mapping  $t\mapsto P_t^{\frac{\alpha}{2}} y$ is smooth in time everywhere except when approaching  $0$ where it is only continuous. In particular, from standard (fractional) heat kernel estimates (see Corollary \ref{cor: heat kernel estimates} in the Appendix)  we know that for $s\leq t \in [0,T]$, any $\theta\in [0,1]$ and $\rho\in (0,\alpha]$ the following inequality holds
\begin{equation}\label{eq: cont sing ineq}
    \|(P_t^{\frac{\alpha}{2}}-P^{\frac{\alpha}{2}}_s)y\|_{\cC^{\varsig+\alpha\rho}} \lesssim \|y\|_{\cC^\kappa}|t-s|^\theta s^{-\theta-\rho}.
\end{equation}
Note the special case when $\rho=0$. Then $S_t$ is a linear operator on $\cC^\kappa$ which is $\theta$-H\"older continuous on an interval $[\eps,T]\subset [0,T]$ for any $\eps>0$ and $\theta\in [0,1]$, but might only continuous when approaching the point  $t=0$.  We therefore need to extend the concept of H\"older spaces in order to take into account this type of loss of regularity near the origin.

\begin{defn}\label{def:explosive-norm}
 Let $E$ be a Banach space. We define the space of continuous paths on $(0,T]$ which are H\"older of order $\varsig\in (0,1)$ on $(0,T]$ and only continuous at zero  in the following way
\begin{equation*}
    \scC^{\varsig}_TE :=\{y:[0,T]\rightarrow E \, | \,  [y]_{\varsig}<\infty \}
\end{equation*}
where we define the semi-norm
\begin{equation*}
    [y]_{\varsig}:=\sup_{s\leq t\in [0,T];\,\zeta\in [0,\varsig]} \frac{|y_{s,t}|_E}{|t-s|^\zeta s^{-\zeta}}.
\end{equation*}
The space $\scC^{\varsig}_TE $ is a Banach space when equipped with the norm $\| y\|_{\varsig}:= |y_0|_{E}+[y]_{\varsig}$.
\end{defn}

Singular H\"older type spaces introduced above has recently  been extensively studied in \cite{bellingeri2020singular}. The space introduced above can be seen as a special case of the more general spaces considered there. 
We use the convention of taking supremum over $\theta\in [0,\varsig]$, as this will make computations in subsequent sections simpler. It is well known that if a function $y\in \cC^\gamma_TE$ for some $\gamma\in [0,1)$, then $y\in \cC^\theta_T E$ for all $\theta\in [0,\gamma]$ (note that $\theta=0$ implies that $y$ is bounded, which is always true for H\"older continuous functions on bounded domains). Thus it follows that  also
\begin{equation*}
    \sup_{s\leq t\in [0,T];\, \theta\in [0,\gamma]} \frac{|y_{s,t}|_E}{|t-s|^\theta} <\infty.
\end{equation*}

\begin{rem}\label{Holder implies sinngular H\"older}
It is readily seen that $\scC^{\varsig}_TE$ consists of all functions which are H\"older continuous on $(0,T]$, but is only continuous in the point $\{0\}$. Note therefore that the following  inclusion $\cC^\varsig_TE\subset \scC^{\varsig}_TE$ holds. Indeed,  for any $s,t\in [0,T]$
\begin{equation*}
  \frac{|y_{s,t}|_E}{|t-s|^\varsig}=  \frac{s^{-\varsig}|y_{s,t}|_E}{s^{-\varsig}|t-s|^\varsig}  
\end{equation*}
and thus in particular, 
\begin{equation*}
     \frac{|y_{s,t}|_E}{s^{-\varsig}|t-s|^\varsig} \leq T^{\varsig}  \frac{|y_{s,t}|_E}{|t-s|^\varsig}.
\end{equation*}
\end{rem}
\begin{rem}
It may be instructive for the reader to keep in mind that in subsequent sections, we will take the Banach space $E$ to be the Besov-H\"older space $\cC^\beta(\RR^d)$  or $\cC^\beta(\TT^d)$ for some $\beta\in \RR$ and $d\geq 1$.
\end{rem}

We will throughout this section work with general Volterra type operators satisfying certain regularity assumptions. We therefore give the following working hypothesis. 

\begin{hyp}\label{kernel hypothesis}
For each $t\in [0,T]$, let $S_t\in \cL(E)$ be a linear operator on $E$ satisfying the following three regularity conditions for some $\rho> 0$ and any $\theta,\theta'\in [0,1]$ and any $0\leq s \leq t \leq \tau' \leq \tau\leq T$
\begin{equation*}
    \begin{aligned}
       {\rm (i)}&\quad   &|S_tu|_E&\lesssim t^{-\rho} |u|_E
         \\
      {\rm (ii)}&\quad  &|(S_t-S_s)u|_E &\lesssim (t-s)^\theta s^{-\theta-\rho}|u|_E
         \\
        {\rm (iii)}&\quad &\Big|\big((S_{\tau-t}-S_{\tau-s})-(S_{\tau'-t}-S_{\tau'-s}) \big )u \Big|_E&\lesssim (\tau-\tau')^{\theta'}(t-s)^\theta(\tau'-t)^{-\theta-\theta'-\rho} |u|_E
    \end{aligned}
\end{equation*}
We then say that $t\mapsto S_t$ is a $\rho$--singular operator. 

\end{hyp}

\subsection{Non-linear Young-Volterra integration}\label{sec:sewing}

We are now ready to construct a non-linear Young-Volterra integral in Banach spaces. If $X : [0,T]\times E \to E$ is a smooth function in time, and $(P_t)_{t\in[0,T]}$ is a (nice) linear operator of $E$ and if $(y_t)_{t\in[0,T]}$ is a continuous path from $[0,T]$ to $E$ itself, it is quite standard to consider integrals of the following form :
\[\int_0^t P_{t-r} \dot{X}_r(y_r) \dd r.\]
The aim this part is to extend the notion of integral for non-smooth drivers $X$, and to solve integral equations using this extension of the integral; The notion of $\rho$--singular operator will be useful in the following. 

\begin{lem}\label{integral lemma}
Consider parameters $\gamma>\frac{1}{2}$, $0\leq \rho \leq \gamma $ and $0<\varsig<\gamma - \rho$ and assume $\varsig+\gamma>1$. Let $E$ be a Banach space and let $y\in \scC^{\varsig}_T E$. Suppose that $X:[0,T]\times E \rightarrow E $  satisfies for any $x,y\in E$ and $s\leq t\in [0,T]$
\begin{equation*}
    \begin{aligned}
    {\rm (i)}\qquad &|X_{s,t}(x)|_{E}\lesssim H(|x|_E) |t-s|^\gamma
    \\
    {\rm (ii)}\qquad &|X_{s,t}(x)-X_{s,t}(y)|_{E}\lesssim H(|x|_E\vee |y|_E)|x-y|_{E}|t-s|^\gamma
    \end{aligned}
\end{equation*}
where $H$ is a positive locally bounded function on $\RR_+$. 
Let  $(S_t)_{t\in[0,T]}\in \cL(E)$ be a $\rho$--singular linear operator on $E$, satisfying Hypothesis \ref{kernel hypothesis}.
We then define the non-linear Young-Volterra integral by 
\begin{equation}\label{eq:theta map}
   \Theta (y)_{t}:=\lim_{\substack{\cP\in\cP([0,t])\\|\cP|\rightarrow 0}} \sum_{[u,v]\in \cP} S_{t-u} X_{u,v}(y_u).
\end{equation}
 The integration map $\Theta$ is a continuous non-linear operator from $ \scC^{\varsig}_T E\rightarrow \scC^{\varsig}_T E$, and there exists an $\eps>0$  such that the following inequality holds 
\begin{equation}\label{eq:bound on theta int }
    |\Theta(y)_t-\Theta(y)_s|_{E} \lesssim \sup_{0\leq z\leq \|y\|_\infty} H(z) (1+[y]_\varsig) (t-s)^{\varsig}T^{\gamma-\rho + \varsig}, 
\end{equation}
  Furthermore, for a linear operator $A\in \cL(E)$, the following commutative property holds
\begin{equation*}
    A\Theta(y)_t=\lim_{|\cP|\rightarrow} \sum_{[u,v]\in \cP} A\, S_{t-u} X_{u,v}(y_u). 
\end{equation*}
\end{lem}

\begin{proof}
Let us first assume that for any $0\leq s \leq t \leq \tau' \leq \tau \leq T$ the following operator is well-defined :
\begin{equation*}
   \Theta_s^t (y)_{\tau}:=\lim_{\substack{ \cP \in \cP([s,t]) \\ |\cP|\to 0}} \sum_{[u,v]\in \cP} S_{\tau-u} X_{u,v}(y_u). 
\end{equation*}
Note that in this setting we have 
\[\Theta(y)_t = \Theta_0^t (y)_t \]
and the increment satisfies
\[\Theta(y)_{s,t} = \Theta_s^t(y)_t + \Theta_0^s(y)_{s,t}.\]
Hence, in order to have the bound \eqref{eq:bound on theta int }, it is enough to have a bound on $\Theta_s^t (y)_{\tau}$ and on $\Theta_s^t (y)_{\tau',\tau}$.

To this end, we will begin to show the existence of the integrals $\Theta_s^t (y)_\tau$ and $\Theta_s^t(y)_{\tau',\tau}$ together with the suitable bounds. Both these terms are constructed in the same way, however since 
\begin{equation}\label{eq:theta map bis}\Theta_s^t(y)_{\tau',\tau} = \lim_{\substack{\cP\in\cP([s,t])\\|\cP|\to 0}} \sum_{[u,v] \in \cP} \left(S_{\tau - u} - S_{\tau' - u}\right) X_{u,v}(y_u),\end{equation}
this term is a bit more involved as it has an increment of the kernel $P$ in the summand. We will therefore show existence as well as a suitable bound for this term and leave the specifics of the first term as a simple exercise for the reader.
Everything will be proven in a similar manner as the sewing lemma from the theory of rough paths. More specifically, the recently developed Volterra sewing lemma from \cite{harang2019tindel} provides the correct techniques to this specific setting.
The uniqueness, and additivity (i.e. that $ \Theta(y)_{s,t} = \Theta_s^t (y)_t + \Theta_0^s(y)_{s,t}$)  of the mapping follows directly from the standard arguments given for example in \cite{harang2019tindel} or \cite[Lem. 4.2]{Friz2014}.

Consider now a dyadic partition $\cP^n$ of $[s,t]$ defined iteratively such that $\cP^0=\{[s,t]\}$, and for $n\geq 0$
\begin{equation*}
    \cP^{n+1} :=\bigcup_{[u,v]\in \cP^n} \{[u,m],[m,v]\},
\end{equation*} 
where $m:=\frac{u+v}{2}$. It follows that $\cP^n$ consists of $2^n$ sub-intervals $[u,v]$, each of length $2^{-n}|t-s|$. Define the approximating sum 
\begin{equation}\label{eq: I_n}
    \cI_{n}:=\sum_{[u,v]\in \cP^n} (S_{\tau-u}-S_{\tau'-u}) X_{u,v}(y_u),
\end{equation}
and observe that for $n\in \NN$, we have 
\begin{equation}\label{eq:diff In}
    \cI_{n+1}-\cI_{n}=-\sum_{[u,v]\in \cP^n} \delta_{m}\left[(S_{\tau-u}-S_{\tau'-u}) X_{u,v}(y_u)\right],
\end{equation}
where $m=\frac{u+v}{2}$ and for a two variable function $f$, we use that   $\delta_m f_{u,v}:=f_{u,v}-f_{u,m}-f_{m,v}$.
By elementary algebraic manipulations we see that 
\begin{multline}\label{delta on inc}
  \delta_{m}\left[(S_{\tau-u}-S_{t-u}) X_{u,v}(y_u)\right]
    \\
    = (S_{\tau-u}-S_{\tau'-u}) (X_{m,v}(y_u)-X_{m,v}(y_m))+(S_{\tau-u}-S_{\tau'-u}-S_{\tau-m}+S_{\tau'-m}) X_{m,v}(y_u). 
\end{multline}
We first investigate the second term on the right hand side above. 
Invoking {\rm (iii)} of Hypothesis \ref{kernel hypothesis} and invoking assumption {\rm (i)} on $X$, we observe that for any $\theta,\theta'\in [0,1]$
\begin{multline*}
    |(S_{\tau-u}-S_{t-u}-S_{\tau-m}+S_{t-m}) X_{m,v}(y_u)|_{E}
    \\
    \lesssim H(|y_u|_E) |\tau-\tau'|^{\theta'}|\tau'-m|^{-\theta'-\rho-\theta}|m-u|^\theta|v-m|^\gamma.
\end{multline*}
Let us now fixe $\theta' = \varsig \in [0,\gamma-\rho)$, and choose $\theta\in [0,1]$ such that $\gamma+\theta>1$ and $\theta +\varsig+\rho<1$. Note that this is always possible due to the fact that $\gamma-\rho-\varsig>0$. Furthermore, we note that for any partition $\cP$ of $[s,t]$ we have \begin{equation}\label{eq: sum to int ineq}
    \sum_{[u,v]\in \cP} |\tau'-m|^{-\theta-\rho-\varsig}|v-m|\lesssim \int_s^t |\tau'-r|^{-\theta-\rho-\varsig}\dd r \lesssim |t-s|^{1-\varsig-\rho-\theta}, 
\end{equation}
where we have used that $m=(u+v)/2$.
From this, it follows that  for any $\theta\in [0,\varsig]$ the following inequality holds 
\begin{multline}\label{eq: sum over rec inc}
     \sum_{[u,v]\in \cP^n}   |(S_{\tau-u}-S_{t-u}-S_{\tau-m}+S_{t-m}) X_{m,v}(y_u)|_{E} 
     \\
     \lesssim \sup_{ 0\leq z\leq \|y\|_\infty} H(z) |\cP^n|^{\gamma+\theta-1}|\tau-\tau'|^\varsig |t-s|^{1-\theta-\rho-\varsig}.
\end{multline}

Let us now move on to the first term in  \eqref{delta on inc}. By invoking  the bounds on $P$ from {\rm (ii)} of Hypothesis \ref{kernel hypothesis} and assumption {\rm (ii)} on $X$, we observe that for any $\theta \ge 0$ and any $0\leq \zeta \leq \varsig$,  
\begin{multline*}
    |(S_{\tau-u}-S_{\tau'-u})(X_{m,v}(y_u)-X_{m,v}(y_m))|_{E}
    \\
    \lesssim |\tau-\tau'|^\theta|\tau'-u|^{-\rho-\theta} |m-v|^\gamma |m-u|^\zeta u^{-\zeta} H(|y_u|_E\vee|y_m|_E) [y]_{\varsig}.
\end{multline*}
Similarly as shown in \eqref{eq: sum over rec inc}, we now take $\theta = \zeta = \varsig$ and we consider a sum over a partition $\cP$ over $[s,t]\subset [0,T]$, and see that since $\varsig + \gamma >1$, and for $\rho+\varsig<1$,
\begin{equation*}
    \sum_{[u,v]\in \cP} |\tau-\tau'|^\varsig |\tau'-m|^{-\rho-\varsig}|m-v|^\gamma|m-u|^\varsig u^{-\varsig} 
    \leq |\cP|^{\varsig+\gamma-1}|\tau-\tau'|^\varsig \int_s^t |\tau'-r|^{-\rho-\varsig} r^{-\varsig}\dd r
\end{equation*}
where again $m=(u+v)/2$.
Furthermore, when $s<t<\tau'$, we have
\begin{align*}
    \int_s^t |\tau'-r|^{-\rho-\varsig} r^{-\varsig} \dd r
    \leq & \int_s^t |\tau'-r|^{-\rho-\varsig} r^{-\varsig} \dd r \\
    \leq & (t-s)^{1-\rho-2\varsig} \int_0^1 (1-r)^{-(\rho+\varsig)}r^{-\varsig} \dd r \\
    \lesssim & (t-s)^{1-\rho-2\varsig}.
\end{align*}
We therefore obtain when specifying $\theta = \varsig$,
\begin{multline}\label{eq:sum over inc in y}
   \sum_{[u,v]\in \cP^n} |(S_{\tau-u}-S_{\tau'-u}) (X_{m,v}(y_u)-X_{m,v}(y_m))|_{E}
    \\
    \lesssim |\cP^n|^{\varsig+\gamma-1} |\tau-\tau'|^\varsig|t-s|^{1-\rho-2\varsig}  \sup_{0\leq z\|y\|_\infty} H(z) [y]_{\varsig}.
\end{multline}
Combining \eqref{eq: sum over rec inc} and \eqref{eq:sum over inc in y},  and  using that for $|\cP^n| = 2^{-n} |t-s|$, it follows from \eqref{eq:diff In} that %
\begin{equation*}
    | \cI_{n+1}(s,t)-\cI_{n}(s,t)|_{E}
    \lesssim \sup_{0\leq z \leq \|y\|_{\infty,[s,t]}} H(z) (1+[y]_{\varsig}) 2^{-n(\gamma-(\rho+\varsig))}  (\tau-\tau')^\varsig (t-s)^{\gamma-\rho-\varsig},
\end{equation*}
For $m>n\in \NN$ thanks to  the triangle inequality, and  the estimate above, we get
\begin{equation}\label{I m n diff}
    \| \cI_{m}(s,t)-\cI_{n}(s,t)\|_{E}
    \lesssim \sup_{0\leq z\|y\|_{\infty,[s,t]}} H(z) (1+[y]_{\varsig})  (\tau-\tau')^\varsig (t-s)^{\gamma-\rho-\varsig} \psi_{n,m}, 
\end{equation}
where $\psi_{n,m}=\sum_{i=n}^m 2^{-i(\gamma-(\rho + \varsig))}$, and it follows that $\{\cI_n\}_{n\in \NN}$ is Cauchy in $E$. It follows that there exists a limit $\cI=\lim_{n\rightarrow \infty} \cI_n$ $E$. Moreover, from  \eqref{I m n diff} we find that the following inequality holds
\begin{equation}\label{qe:bound ineq 1}
    |\cI-(S_{\tau-s}-S_{\tau'-s}) X_{s,t}(y_s)|_{E} 
    \lesssim \sup_{0\leq z\leq  \|y\|_{\infty}} H(z)(1+[y]_{\varsig})  (\tau-\tau')^\varsig (t-s)^{\gamma-\rho-\varsig} \psi_{0,\infty}.
\end{equation}
the proof that $\cI$ is equal to $\Theta_{s}^{t}(y)_{\tau,\tau'}$ defined in \eqref{eq:theta map bis} (where in particular $\Theta(y)$ is defined independent of the partition $\cP$ of $[s,t]$) follows by standard arguments for the sewing lemma, see \cite{harang2019tindel} in the Volterra case. Finally, we get for any $0 \leq s\leq t \leq \tau' \leq \tau \leq T$,
\begin{equation}\label{eq:bound theta 1 remi}|\Theta_s^t(y)_{\tau',\tau} - (S_{\tau - s} - S_{\tau'-s})X_{s,t}(y_s) |_E \lesssim \sup_{0 \leq z \leq \|y\|_\infty} H(z) T^{\gamma -\rho + \varsig} |t-s|^\varsig.
\end{equation}

Using the same techniques, in order to prove that $\Theta_{s}^{t}(y)_t$ exists for all $0\leq s \leq t \leq \tau \leq T$ one has to control
\[S_{\tau-m}(X_{m,v}(y_m) - X_{m,v}(y_u) ) + (S_{\tau-u} - S_{\tau-m})X_{m,v}(y_u). \]
Performing exactly the same computation, we have
\begin{equation}\label{eq:bound theta 2 remi}|\Theta_{s}^{t}(y)_\tau- S_{\tau-s}X_{s,t}(y_s) |_E \lesssim \sup_{0\leq z \leq \|y\|_{\infty}} H(z)T^{\gamma-\varsig+\rho}|t-s|^{\varsig} (1+[h]_\varsig).
\end{equation}
Combining \eqref{eq:bound theta 1 remi} and \eqref{eq:bound theta 2 remi}, 
and using the fact that 
\[\Theta(y)_t - \Theta(y)_s = \Theta_{s,t}(y)_t + \Theta_0^s(y)_{s,t}\]
\begin{multline*}
     |\Theta(y)_{t,s}-S_{t-s}X_{s,t}(y_s) - (S_{t}-S_{s})X_{0,s}(y_0)|_{E}\\ \leq 
     |\Theta_s^t (y)_t-(S_{t} - S_s)X_{s,t}(y_s)|_{E}
     +|\Theta_0^s(y)_{s,t}-S_{t-s}\xi X_{0,s}(y_0)|_{E}
\end{multline*}
it is readily checked that the following inequality holds
\begin{equation*}
    |\Theta(y)_t-\Theta(y)_s -S_{t-s}X_{s,t}(y_s) - (S_{t}-S_{s})X_{0,s}(y_0)|_{E} \lesssim \sup_{0\leq \leq \|y\|_\infty}H(z) (1+[y]_{\varsig}) T^{\gamma-\rho + \varsig} (t-s)^{\varsig}.
\end{equation*}
Finally note that since $\varsig < \gamma - \rho$,
\[|S_{t-s}X_{s,t}(y_s)|_E \lesssim H(|y_s|) |t-s|^{\gamma-\rho} \lesssim H(|y_s|) T^{\gamma-\rho+\varsig}|t-s|^\varsig,\]
and
\[|(S_{t}-S_{s})X_{0,s}(y_0)|_{E} \lesssim H(|y_0|) s^{\gamma-\rho - \varsig}|t-s|^\varsig.\]
For the last claim, if $A\in \cL(E)$ is a linear operator, then one re-define $\cI_n$ in \eqref{eq: I_n} to be given as 
\begin{equation*}
    \cI_n(A):= \sum_{[u,v]\in \cP^n} A(S_{\tau-u}-S_{t-u})X_{u,v}(y_u), 
\end{equation*}
and by linearity of $A$ we see that $\cI_n(A)=A\cI_n$. 
Taking the limits, using the above established inequalities, we find that 
$\lim_{n\rightarrow \infty}\| \cI_n(A)-A\cI_n\|_{E} =0$. 
\end{proof}

With the construction of the non-linear Young Volterra integral, we will in later applications need certain stability estimates.

\begin{prop}[Stability of $\Theta$]\label{prop: stability of theta}  Let $\gamma,\varsig,\rho$ be given as in of Lemma \ref{integral lemma}. Assume that for $i=1,2$, $X^i:[0,T]\times E\rightarrow E$  satisfies for any $x,y\in E$ and $s\leq t\in [0,T]$
\begin{equation}\label{eq: conditions for stability}
    \begin{aligned}
    &{\rm (i)}\qquad &|X^i_{s,t}(x)|_{E}+\|\nabla X^i_{s,t}(x)|_{\cL(E)}&\lesssim H(|x|_E) |t-s|^\gamma
    \\
    &{\rm (ii)}\qquad &|X^i_{s,t}(x)-X^i_{s,t}(y)|_{E}&\lesssim H(|x|_E\vee |y|_E)|x-y|_{E}|t-s|^\gamma
    \\
    &{\rm (iii)}\qquad &|\nabla X^i_{s,t}(x)-\nabla X^i_{s,t}(y)|_{\cL(E)}&\lesssim H(|x|_E\vee |y|_E)|x-y|_{E}|t-s|^\gamma,
    \end{aligned}
\end{equation}
where $H$ is a positive locally bounded function, and $\nabla$ is understood as a linear operator on $E$ in the Fr\'echet sense. 
Furthermore, suppose there exists a positive and locally bounded function $H_{X^1-X^2}$ such that 
\begin{equation}\label{eq: diff bounds in X}
        \begin{aligned}
    &{\rm (i)}\qquad &|X^{1}_{s,t}(x) - X^{2}_{s,t}(x)|_{E}&\lesssim H_{X^1-X^2}(|x|_E) |t-s|^\gamma
    \\
    &{\rm (ii)}\qquad &|(X^{1}_{s,t} - X^{2}_{s,t})(x) -(X^{1}_{s,t}-X^{2}_{s,t})(y)|_{E}&\lesssim H_{X^1-X^2}(|x|_E\vee |y|_E)|x-y|_{E}|t-s|^\gamma
    \end{aligned}
\end{equation}

Let $\Theta^1$ denote the non-linear integral operator constructed in Lemma \ref{integral lemma}  with respect to $X^1$, and similarly let $\Theta^2$ denote the integral operator with respect to $X^2$. 
Then for two paths $y,\tilde{y}\in \scC^{\varsig}_T\cC^{\beta+2\rho}$, 
\begin{multline}\label{eq:stability bound}
    [\Theta^1(y^1)-\Theta^2(y^2)]_{\varsig} \lesssim_P \bigg[\sup_{0\leq z\leq \|y^1\|\vee \|y^2\|} H(z)\left([y^1]_{\varsig}+[y^2]_{\varsig}\right) (|y^1_0-y^2_0|_{E}+[y^1-y^2]_{\varsig}) 
    \\
    +\sup_{0\leq z \leq \|y^1\|_\infty\vee \|y^2\|_\infty }H_{X^1-X^2}(z) [y^1]_{\varsig}\vee[y^2]_{\varsig}\bigg]  T^{\gamma-\rho+\varsig}.
\end{multline}

\end{prop}

\begin{proof}
Define the following two functions
\begin{equation*}\label{eq: diff defined}
    \Theta_{s,t}^1(y^1)-\Theta_{s,t}^2(y^2)=(\Theta_{s,t}^1(y^1)-\Theta_{s,t}^2(y^1))+(\Theta_{s,t}^2(y^1)-\Theta_{s,t}^2(y^2))=:D_{X^1,X^2}(s,t)+D_{y^1,y^2}(s,t). 
\end{equation*}
We treat $D_{X^1,X^2}$ and $D_{y^1,y^2}$ separately, and begin to consider $D_{y^1,y^2}$.  
Since  $X^i$ is differentiable and satisfies {\rm (i)-(iii)} in \eqref{eq: conditions for stability},  for $i=1,2$ we have
\begin{equation*}
    X^i_{s,t}(y^1_s)-X^i_{s,t}(y^2_s)=\cX^i_{s,t}(y^1_s,y^2_s)(y^1_s-y^2_s),
\end{equation*}
where  $\cX^i(y^1_s,y^2_s):=\int_0^1 \nabla X^i_{s,t} (qy^1_s+(1-q)y^2_s)\dd q$.
In order to prove \eqref{eq:stability bound}, we proceed with the exact same strategy as outlined in the proof of Lemma \ref{integral lemma}. That is, we use the same proof as the proof of Lemma \ref{integral lemma} to first prove appropriate bounds for $D_{y^1,y^2}(s,t)$ and similarly for $D_{X^1,X^2}(s,t)$ afterwards. To this end,  changing the integrand in \eqref{eq: I_n} so that  
\begin{equation*}
    \cI_{n}(s,t):=\sum_{[u,v]\in \cP^n[s,t]} (S_{\tau-u}-S_{t-u}) \cX_{u,v}^i(y^1_u,y^2_u)(y^1_u-y^2_u),  
\end{equation*}
we continue along the lines of the proof in Lemma \ref{integral lemma} to show that $\cI^n$ is Cauchy. As the strategy of this proof is identical to that of Lemma \ref{integral lemma} we will here only point out the important differences. 
 By appealing to the condition {\rm (iii)}   in \eqref{eq: conditions for stability},we observe in particular that 
\begin{equation*}
    |\cX^i(y^1_s,y^2_s)-\cX^i(y^1_u,y^2_u)|_{\cL(E)}\lesssim |t-s|^\gamma  \sup_{0\leq z \leq \|y^1\|_{\infty}\vee\|y^2\|_\infty}H(z)\left([y^1]_{\varsig}+[y^2]_{\varsig}\right) .
\end{equation*}
Furthermore, it is readily checked that 
\begin{equation*}
    |y^1_s-y^2_s|_E\lesssim |y^1_0-y^2_0|_E+[y-\tilde{y}]_\varsig T^\varsig.
\end{equation*}
Following along the lines of the proof of Lemma \ref{integral lemma}, one can then check that for $m>n\in \NN$ 
\begin{multline}\label{nr2: I m n diff}
    | \cI_{m}(s,t)-\cI_{n}(s,t)|_{E}
    \\
    \lesssim \sup_{0\leq z \leq \|y^1\|_{\infty}\vee \|y^2\|_\infty } H(z)\left([y^1]_\varsig+[y^2]_\varsig\right) (|y^1_0-y^2_0|_{E}+[y^1-y^2]_{\varsig}T^\varsig)  (\tau-\tau')^\varsig (t-s)^{\gamma-\rho-\varsig} \psi_{n,m},
\end{multline}
where $\psi_{n,m}$ is defined as below \eqref{I m n diff}. 
With this inequality at hand, the remainder of the proof can be verified in a similar way as in the proof of Lemma \ref{integral lemma}, and we obtain from this lemma that 
\begin{equation*}
    \|D_{y^1,y^2}\|_{\scC^\varsig_T E} \lesssim C\sup_{0\leq z \leq \|y^1\|_{\infty}\vee\|y^2\|_\infty }(z)\left([y^1]_\varsig+[y^2]_\varsig\right) (\|y^1_0-y^2_0\|_{E}+[y^1-y^2]_{\varsig}T^\varsig). 
\end{equation*}

Next we move on to prove a similar bound of $D_{X^1,X^2}$ as defined in \eqref{eq: diff defined}. Set $Z=X^1-X^2$. By \eqref{eq: diff bounds in X} it follows that $Z$ satisfies  the conditions of Lemma \ref{integral lemma}, and then from \eqref{eq:bound on theta int } it follows that 
\begin{equation*}
    [D_{X_1,X^2}]_\varsig \lesssim \sup_{0\leq z \leq \|y^1\|_\infty \vee \|y^2\|_\infty} H_{X^1-X^2}(z) [y^1]_\varsig\vee [y^2]_\varsig T^{\gamma - \varsig - \rho}. 
\end{equation*}

\end{proof}

\subsection{Existence and uniqueness}

We begin to prove local existence and uniqueness for an abstract type of equation with  values in a Banach space. The equation in itself does not require the use of the non-linear Young-Volterra integral, and is formulated for general operators $\Theta:[0,T]\times\scC^{\varsig}_T E \rightarrow \scC^{\varsig}_T E$ satisfying certain regularity conditions. We will apply these results in later sections in combination with the non-linear Young-Volterra integral operator $\Theta$ created in the previous section, and thus the reader is welcome to already think of $\Theta$ as being a non-linear Young integral operator as constructed in Lemma \ref{integral lemma}.  

\begin{thm}[Local existence and uniqueness]\label{thm: abs existence and uniqueness} Let $\Theta:[0,T]\times\scC^{\varsig}_T E \rightarrow \scC^{\varsig}_T E$ be a function  which satisfies for $y,\tilde{y}\in \scC^{\varsig}_T E$ and some $\epsilon>0$
\begin{equation}\label{eq:abs cond for ex uni}
\begin{aligned}
     \left[\Theta(y)\right]_\varsig & \leq  C(\|y\|_\infty)(1+[y]_{\varsig})T^\eps
    \\
      [\Theta(y)-\Theta(\tilde{y})]_{\varsig} & \leq C(\|y\|_{\infty}\vee\|\tilde{y}\|_\infty)([y]_\varsig+[\tilde{y}]_\varsig) (|y_0-\tilde{y}_0|_{E}+[y-\tilde{y}]_{\varsig}) T^\eps, 
    \end{aligned}
\end{equation}
where $C$ is a positive and increasing locally bounded function. Consider $p\in \scC^\varsig_T E$ and  let $\tau>0$ be such that 
\begin{equation}\label{eq:tau cond}
\tau \leq \left[4(1+[p]_\varsig) C(1+|p_0|+[p]_\varsig)\right]^{-\frac{1}{\eps}} .    
\end{equation}
Then there exists a unique solution to the equation 
\begin{equation}\label{eq: abs equation}
    y_t=p_t +\Theta(y)_{t}, \qquad p\in \scC^\varsig_T E
\end{equation}
in $\cB_{\tau}(p)$, where $\cB_{\tau}(p)$ is a unit ball in $\scC^{\varsig}_\tau E$, centered at $p$. 
\end{thm}

\begin{proof}
To prove existence and uniqueness, we will apply a standard fixed point argument. Define the solution map $\Gamma_\tau:\scC^{\varsig}_\tau E \rightarrow \scC^{\varsig}_\tau E$ given by 
$$
\Gamma_\tau(y):=\{p_t+\Theta(y)_t|\, t\in [0,\tau]\}.
$$
Since $\Theta:\scC^{\varsig}_\tau E\rightarrow \scC^{\varsig}_\tau E $ and $p\in \scC^{\varsig}_\tau E$ it follows that $\Gamma_\tau(y)\in  \scC^{\varsig}_\tau E$. We will now prove that the solution map $\Gamma_\tau$ is an invariant map and a contraction on a unit ball $\cB_\tau(p)\subset \scC^\varsig_\tau E$ centered at $p\in \scC^{\varsig}_\tau E$. In particular, we define 
\begin{equation*}
    \cB_\tau(p):=\{y\in \scC^\varsig_\tau |\, y_t=p_t+z_t,\,{\rm with }\,\,z\in \scC^\varsig_\tau E,\,\, z_0=0,\,\,[y-p]_\varsig\leq 1\}. 
\end{equation*}

We begin with the invariance. From the first condition in \eqref{eq:abs cond for ex uni}, it is readily checked that for $y\in \cB_\tau(P)$  
\begin{equation}\label{eq:boundedness of gamma}
    [\Gamma_\tau(y)-p]_{\varsig} \leq C(1+|p_0|+[p]_\varsig)(1+[y]_{\varsig})\tau^\eps,  
\end{equation}
where we have used that for  $y\in \cB_\tau(P)$ ,  $\|y\|_\infty \leq 1+ |p_0|+[p]_\varsig$
Choosing a parameter ${\tau_1}>0$ such that  
$$
\tau_1\leq (2C(1+|p_0|+[p]_\varsig)^{-\frac{1}{\eps}}
$$
it follows that $\Gamma_{\tau_1}(\cB_{\tau_1}(p))\subset \cB_{\tau_1}(p)$and we say that
$\Gamma_{\tau_1}$ leaves the ball $\cB_{\tau_1}(p)$ invariant.

Next, we prove that $\Gamma_\tau$ is a contraction on $\cB_{\tau}(p)$. From the second condition in \eqref{eq:abs cond for ex uni}, it follows that for two elements $y,\tilde{y}\in\cB_\tau(p)$ we have 
\begin{equation*}
    [\Gamma(y)-\Gamma(\tilde{y})]_{\varsig} \leq  2(1+[p]_\varsig) C(1+|p_0|+[p]_\varsig) [y-\tilde{y}]_{\varsig}\tau^\eps,
\end{equation*}
where we have used that $y_0=\tilde{y}_0$, and 
$$
[y]_\varsig\vee [\tilde{y}]_\varsig \leq 1+[p]_\varsig\quad  {\rm and}\quad \|y\|_\infty\vee \|\tilde{y}\|_\infty \leq 1+|p_0|+[p]_\varsig. 
$$
Again, choosing a parameter $\tau_2>0$ such that  
$$
\tau_2\leq \left(4(1+[p]_\varsig)C(1+|p_0|+[p]_\varsig)\right)^{-\frac{1}{\eps}}. 
$$
it follows that 
$$
 [\Gamma(y)-\Gamma(\tilde{y})]_{\varsig;\tau_2} \leq \frac{1}{2} [y-\tilde{y}]_{\varsig;\tau_2}. 
 $$
Since  $\tau_2\leq \tau_1$, we  conclude that 
the solution map $\Gamma_{\tau_2}$ both is an invariant map and a contraction on the unit ball $\cB_{\tau_2}(p)$. It follows by Picard-Lindl\"ofs fixed point theorem that a unique solution to \eqref{eq: abs equation} exists in $\cB_{\tau_2}(p)$. 

\end{proof}

The next theorem shows that if the locally bounded function $C$ appearing in the conditions on $\Theta$ in \eqref{eq:abs cond for ex uni} of Theorem \ref{thm: abs existence and uniqueness}, is uniformly bounded, then there exists a unique global solution to   \eqref{eq: abs equation}. 

\begin{thm}[Global existence and uniqueness]\label{Global Existence}
Let $\Theta:[0,T]\times\scC^{\varsig}_T E \rightarrow \scC^{\varsig}_T E$ satisfy \eqref{eq:abs cond for ex uni} for a positive,  globally bounded function $C$, i.e. there exists a constant $M>0$ such that $\sup_{x\in\RR_+} C(x)\leq M$. Furthermore, suppose $\Theta$ is time-additive, in the sense that $\Theta_t=\Theta_s+\Theta_{s,t}$ for any $s\leq t\in [0,T]$.  Then for any $p\in \scC^\varsig_T E$ there exists a unique solution $y\in \scC^\varsig_T E$ to the equation 
\begin{equation*}
    y_t=p_t +\Theta(y)_{t},\quad t\in [0,T].  
\end{equation*}
\end{thm}
 
 \begin{proof}
 By Theorem \ref{thm: abs existence and uniqueness} we know that there exists a unique solution to \eqref{eq: abs equation} on an interval $[0,\tau]$, where $\tau$ satisfies \eqref{eq:tau cond}, and $C$ is replaced by the bounding constant $M$, i.e.
 \begin{equation}\label{eq:global tau}
     \tau \leq \left[4(1+[p]_\varsig)M\right]^{-\frac{1}{\eps}}. 
 \end{equation}
By a slight modification of the proof in Theorem \ref{thm: abs existence and uniqueness} it is readily checked that the existence and uniqueness of 
$$
y_t=p_t+\Theta_{a,t}(y),\quad t\in [a,a+\tau],
$$
holds on any interval $[a,a+\tau]\subset [0,T]$, i.e. the solution is constructed in $\cB_{[a,a+\tau]}(p)$.

Now, we want iterate solutions to \eqref{eq: abs equation} to the domain $[0,T]$, by "gluing together" solutions on the integrals $[0,\tau],[\tau,2\tau]...\subset [0,T]$.  Using the time-additivity property of $\Theta$, note that for $t\in [\tau,2\tau]$, we have
\begin{equation*}
    y_t = p_t+\Theta_t(y)=p_t+\Theta_a(y)+\Theta_{a,t}(y). 
\end{equation*}
Thus, set $\tilde{p}_t=p_t+\Theta_a(y|_{[0,\tau]})$, where $y|_{[0,\tau]}$ denotes the solution to \eqref{eq: abs equation} restricted to $[0,\tau]$. Note that 
\begin{equation*}
    [\tilde{p}]_\varsig=[p]_\varsig, 
\end{equation*}
since the H\"older seminorm is invariant to constants, and $t\mapsto \Theta_a(y|_{[0,\tau]})$ is constant. Therefore, there exists a unique solution to $\eqref{eq: abs equation}$ in $\cB_{[\tau,2\tau]}(\tilde{p})$ where $\tau$ is the same as in \eqref{eq:global tau}. We can repeat this to all intervals $[k\tau,(k+1)\tau]\subset [0,T]$. 
At last, invoking the scalability of H\"older norms (see \cite[Exc. 4.24]{Friz2014}),  it follows that there exists a unique solution to \eqref{eq: abs equation} in $\scC^\varsig_T E$.

 \end{proof}

\section{Averaged fields}\label{averaged fields}

We give here a quick overview of the concept of averaged fields and averaging operators. 
We begin with the following definition: 
\begin{defn}\label{def of avg operator}
Let $\omega$ be a measurable path from $[0,T]$ to $\RR^d$, and let $g\in \cS'(\RR^d;\RR)$. We define the average of $g$ against $\omega$ as the element of $C^0\big([0,T];\cS'(\RR^d;\RR)\big)$ defined for all $s\leq t\in [0,T]$ and all test functions $\phi$ by
\[\big\la \phi,T^{\omega}_{s,t} g \big\ra = \int_s^t \la \phi(\cdot - \omega_r),g\ra\dd r.\]
\end{defn}

Introduced in the analysis of regularization by noise in \cite{Catellier2016}, the concept of averaged fields and averaging operators is by now a well studied topic. For example, the recent analysis of Galeati and Gubinelli \cite{galeati2020noiseless} provides a good overview of the analytic properties in the context of regularization by noise. See also \cite{harang2020cinfinity,galeati2020prevalence,galeati2020regularization} for further details on probabilistic and analytical aspects of averaged fields and averaging operators, and their connection to the concept of occupation measures. 
In the current article, we investigate these operators from an infinite-dimensional perspective in order to apply them in the context of (S)PDEs, and thus some extra considerations needs to be taken into account.  In addition, we include in section \ref{sec: Averaging operators with Levy noise} a construction of the averaged field associated to a fractional L\'evy process. The regularizing properties of Volterra-L\'evy processes was recently investigated in  \cite{harang2020regularity}, where the authors constructed the averaged field using the concept of local times. The construction given in the current article provides better regularity properties of the resulting averaged field than those obtained in \cite{harang2020regularity}, but comes at the cost of lost generality. More precisely, constructing the averaged field associated to a  measurable stochastic process $\omega:\Omega \times [0,T]\rightarrow \RR^d$ acting on a distribution $b\in \cS'(\RR^d)$,  the exceptional set $\Omega'\subset \Omega$ on which the function $T^{\omega} g$ is a sufficiently regular field (say, H\"older in time and differentiable in space..)  depends on  $b\in \cS'(\RR^d)$, i.e. $\Omega'=\Omega'(b)$. Thus choosing one construction or the other depends on the problem at hand, and which  properties  are important to retain. We will not investigate these differences in more details in this article, and will view the analysis of the SPDE from a purely deterministic point of view. 
Let us give the following statement which can be viewed as a short summary of some of the results appearing in \cite{galeati2020noiseless} and \cite{harang2020cinfinity}: 

\begin{prop}\label{prop: existecnce of avg operator }
There exists a $\delta$-H\"older continuous path $\omega:[0,T]\rightarrow \RR$ such that for any given $g\in \cC^\eta$ with $\eta>3-\frac{1}{2\delta}$, the corresponding averaged field $T^\omega g$ is contained in $\cC^\gamma_T\cC^\kappa$ for some $\kappa\geq 3$ and $\gamma>\frac{1}{2}$. Moreover, there exists a continuous $\omega:[0,T]\rightarrow \RR$  such that for any  $g\in \cS'(\RR)$, the averaged field $T^\omega g$ is contained in $\cC^\gamma_T\cC^\kappa(w)$ for some weight $w:\RR\rightarrow \RR_+$ and any $\gamma\in (\frac{1}{2},1)$ and any  $\kappa\in \RR$. 
\end{prop}
\begin{proof}
The first statement can be seen as a simple version of \cite[Thm. 1]{galeati2020noiseless}. The second is proven in \cite[Prop. 24]{harang2020cinfinity}. 
\end{proof}
\begin{rem}
In fact, the statement of \cite[Thm. 1]{galeati2020noiseless} is much stronger; given a $g\in \cC^\eta$ with $\eta\in \RR$, then  for almost all  $\delta$-H\"older continuous paths $\omega:[0,T]\rightarrow \RR$ such that $\eta>3-\frac{1}{2\delta}$, the averaged field $T^\omega g\in \cC^\gamma_T\cC^3$ for some $\gamma>\frac{1}{2}$. Similarly, it is stated that almost all continuous paths $\omega$ are infinitely regularizing in the sense that for any $g\in \cC^\eta$ with $\eta\in \RR$,  $T^\omega g\in \cC^\gamma \cC^\kappa$ for any $\kappa \in \RR$. The "almost surely" statement here is given through the concept of prevalence. We refrain from writing proposition \ref{prop: existecnce of avg operator } in the most general way in order to avoid going into details regarding the concept of prevalence here. We therefore refer the reader to \cite{galeati2020noiseless} for more details on this result and this concept. 
\end{rem}

\begin{rem}
The statement in Proposition \ref{prop: existecnce of avg operator } can also be generalized to measurable paths $\omega:[0,T]\rightarrow \RR$. Indeed, in \cite{harang2020regularity} the authors show that there exists a class of measurable Volterra-L\'evy processes, which provides a regularizing effect, similar to that of Gaussian processes, and a statement similar to that of Proposition \ref{prop: existecnce of avg operator } can be found there. 
\end{rem}

in all the following we will use notions of (weighted)-Besov space. For a recap on weighted Lebesgue and Besov spaces, as long as a recap on standard (fractional)-Schauder estimates see Appendix Sections \ref{sub:weighted_lebesgue} and \ref{sub:weighted_besov}.  

We continue with some properties which will be useful in later analysis.  

\begin{prop}
Let $w$ be an admissible weight, $\kappa \in \RR$ and $1\leq p,q\leq +\infty$. Let $f \in B^{\kappa}_{p,q}(w)$. For all $j\ge -1$ and all $0\leq s \leq t \leq T$,
\[\Delta_j T^{\omega}_{s,t} g =  T^{\omega}_{s,t} (\Delta_j g),\]
where $\Delta_j$ denotes the standard Paley-Littlewood block (see Appendix \ref{sub:weighted_besov}).

In particular for all $\eps,\delta>0$, and for all $f\in B^{\kappa}_{p,q}(w)$, and for $\cS_k = \sum_{j={-1}^k }\Delta_j$,
\[ T^\omega_{s,t} (\cS_k g) \underset{k\to \infty}{\to} T^\omega_{s,t} g \quad \text{in} \quad \cS' \quad \text{and in} \quad B^{\kappa-\eps}_{p,q}(\llrr^{\delta}w).\]
Suppose that $g$ in a measurable locally bounded function, then $T^\omega_{s,t}f$ is also a measurable function one has 
\[T^\omega_{s,t} g(x) = \int_s^t g(x + \omega_r) \dd r. \]
\end{prop}

\begin{proof}
Let us remind that the topology on $\cS$ is the one generated by the family of semi-norm
\[\cN_n(\phi) = \sum_{|k|,|l| \leq n} \sup_{x\in\RR^d} |x^k \partial^l \phi(x)|,\]
where $k$ and $l$ are multi-indices. Furthermore, for any distribution $f\in \cS'$, there exists $n\ge 0 $ and $C>0$ such that 
\[|\la \phi , g \ra | \leq C \cN_n(\phi).\]
Let $\omega$ be a measurable path from $[0,T]$ to $\RR^d$, then since $\phi$ is bounded, there is a constant $C_1>0$ such that
\[\cN_n\big(\phi(\cdot - w_r)\big) \leq C \cN_n\big(\phi\big).\]
In particular
\[\Delta_j T^{\omega}_{t} g (x) = \left\la K_j(x-\cdot), T^{\omega}_{s,t}f \right\ra = \int_0^t \left\la K_j\big(x+\omega_r)-\cdot\big), g \right\ra \dd r = \int_0^t \Delta_j g(x+\omega_r) \dd r.\]
Thanks to the previous remark, we are allowed to perform any Fubini arguments, and for all $\phi\in\cS$,
\[\big\la \phi , \Delta_j T^{\omega}_{t} g \big\ra = \int_{\RR^d} \phi(x) \int_0^t \Delta_j g(x+\omega_r) \dd r \dd x = \int_0^t \int_{\RR^d} \phi(x-\omega_r) \Delta_j g(x) \dd x \dd r,\]
which ends the proof by using properties of weighted-Besov spaces.
\end{proof}

For the purpose of this section, we will fix a distribution $g\in \cS'(\RR^d)$ and a measurable path $\omega:[0,T]\rightarrow \RR^d $ with the property that there exists an averaged field $T^\omega g:[0,T]\times \RR^d\rightarrow \RR^d$ defined as in Definition \ref{def of avg operator} which is H\"older continuous in time and three times locally differentiable in space. More specifically, we will assume that $T^\omega g\in \cC^\gamma_T\cC^\kappa(w)$ for some admissible weight function $w:\RR^d\rightarrow \RR_+\setminus \{0\}$.
\\

Our first goal is to show how the averaged field $T^w g$ can be seen as a function from $[0,T]\times \cC^\beta\rightarrow \cC^\beta$ for some $\beta\geq0$.

\begin{prop}\label{prop: inf dim reg of avg op}
Let $w$ be an admissible weight (see Section \ref{sub:weighted_lebesgue}). For a measurable path $\omega:[0,T]\rightarrow \RR^d$ and $g\in \cS'(\RR^d)$, suppose $T^\omega g\in \cC^\gamma_T\cC^\kappa(w)$  for some $\gamma>\frac{1}{2}$ and $\kappa\geq 3$. 
Then for all $x,y\in \cC^\beta$ with $\beta\in [0,1)$, we have that 
\begin{equation*}
    \begin{aligned}
         \|T^\omega_{s,t}g(x)\|_{\cC^\beta}\vee \|\nabla T^\omega_{s,t}g(x)\|_{\cC^\beta}  &\leq \sup_{|z| \leq \|x\|_{\cC^\beta}} w^{-1}(z) \|T^\omega g\|_{\cC^\gamma_T\cC^\kappa(w)} |t-s|^\gamma
         \\
         \|T^\omega_{s,t}g(x)-T^\omega_{s,t}g(y)\|_{\cC^\beta} &\leq \sup_{|z| \leq \|x\|_{\cC^\beta}\vee\|y\|_{\cC^\beta}} w^{-1}(z) \|T^\omega g\|_{\cC^\gamma_T\cC^\kappa(w)}\|x-y\|_{\cC^\beta} |t-s|^\gamma
         \\
         \|\nabla T^\omega_{s,t}g(x)-\nabla T^\omega_{s,t}g(y)\|_{\cC^\beta} &\leq \sup_{|z| \leq \|x\|_{\cC^\beta}\vee\|y\|_{\cC^\beta}} w^{-1}(z) \|T^\omega g\|_{\cC^\gamma_T\cC^\kappa(w)}\|x-y\|_{\cC^\beta} |t-s|^\gamma.
    \end{aligned}
\end{equation*}

\end{prop}

\begin{proof}
For $x\in \cC^\beta$ note that for any $\xi\in \RR^d$
$$
|T^\omega_{s,t}g(x(\xi))|\leq w^{-1}(x(\xi))|w(x(\xi))T^\omega_{s,t}g(x(\xi))|\leq \sup_{|z|\leq \|x\|_{\cC^\beta}} w^{-1}(z) \|T^\omega g\|_{\cC^\gamma_T\cC^\kappa}.   
$$
By similar computations using that $T^\omega b$ is  (weighted) differentiable in space, it is straightforward to verify that 
\begin{equation*}
    \|T^\omega_{s,t}g(x)\|_{\cC^\beta}\vee \|\nabla T^\omega_{s,t}g(x)\|_{\cC^\beta}  \leq \sup_{|z| \leq \|x\|_{\cC^\beta}} w^{-1}(z) \|T^\omega g\|_{\cC^\gamma_T\cC^\kappa(w)} |t-s|^\gamma. 
\end{equation*}
Now consider $x,y\in \cC^\beta$. Again  using the differentiability of $T^\omega b$, and elementary rules of calculus we observe that for any $\xi\in \RR^d$
\begin{equation*}
    |T^\omega_{s,t}g(x(\xi))-T^\omega_{s,t}g(y(\xi))|\leq \sup_{|z|\leq \|x\|_{\cC^\beta}\vee \|y\|_{\cC^\beta}} w^{-1}(z) \|T^\omega g\|_{\cC^\gamma_T\cC^\kappa(w)} \|x-y\|_{\cC^\beta}|t-s|^\gamma,
\end{equation*}
where we have used that $|x(\xi)-y(\xi)|\leq \|x-y\|_{\cC^\beta}$ for all $\xi\in \RR^d$.
At last, for $\xi,\xi'\in \RR^d$, observe that 
\begin{equation*}
    \begin{aligned}
       & T^\omega_{s,t}g(x(\xi))-T^\omega_{s,t}g(y(\xi))-T^\omega_{s,t}g(x(\xi'))+T^\omega_{s,t}g(y(\xi')) 
        \\
      &  = \int_0^1 \nabla T^{\omega}_{s,t}g\Big(l\big(x(\xi) - y(\xi) \big) + y(\xi) \Big)\big(x(\xi) - x(\xi') - y(\xi) + y(\xi')\big) \dd l \\
    &+ \int_0^1 \int_0^1 D^2 T^{\omega}_{s,t}g\Big(\Lambda(l,l')\Big)  \Big(l\big(x(\xi) - x(\xi')\big) +(1-l)\big(y(\xi) - y(\xi')\big)\Big) \otimes\big(x(\xi)-y(\xi)\big) \dd l' \dd l,
    \end{aligned}
\end{equation*}
with
\[
\Lambda(l,l') =  l'\Big(l\big(x(z)-x(z')\big)+(1-l)\big(y(z)-y(z')\big)\Big) + lx(z') + (1-l)y(z').
\]
Since $T^\omega g$ is twice (weighted) differentiable, it follows that 
\begin{equation*}
    \|T^\omega_{s,t}g(x)-T^\omega_{s,t} g(y)\|_{\cC^\beta} \leq \sup_{|z|\leq \|x\|_{\cC^\beta}\vee\|y\|_{\cC^\beta}} w^{-1}(z)\|T^\omega g\|_{\cC^\gamma_T\cC^\beta(w)} \|x-y\|_{\cC^\beta}|t-s|^\gamma. 
\end{equation*}
A similar argument for $\nabla T^\omega b$, using that $\kappa\geq 3$, reveals that 
\begin{equation*}
    \|\nabla T^\omega_{s,t}g(x)-\nabla T^\omega_{s,t}g(y)\|_{\cC^\beta} \leq \sup_{|z|\leq \|x\|_{\cC^\beta}\vee\|y\|_{\cC^\beta}} w^{-1}(z)\|T^\omega b\|_{\cC^\gamma_T\cC^\beta(w)} \|x-y\|_{\cC^\beta}|t-s|^\gamma,
\end{equation*}
which concludes this proof. 
\end{proof}

\section{Existence and uniqueness of the mSHE}\label{sec: existence and uniqueness of mSHE}

We are now ready to formulate the multiplicative stochastic heat equation using the abstract non-linear Young integral and the averaged field $T^\omega g$ introduced in the previous section. 
\subsection{Standard multiplicative Stochastic Heat equation with additive noise}
Recall that the mSHE with additive (time)-noise is given in its mild form by
\begin{equation}\label{eq:mSHE particular}
    u_t=P_tu_0+\int_0^t P_{t-s}\xi  g(u_s)\dd s+\omega_t, \quad u_0\in \cC^\beta, 
\end{equation}
where $\xi\in \cC^{-\vartheta}$ for some $0\leq \vartheta <\beta<2-\vartheta$, $P$ denotes the standard heat semi-group acting on functions $u$ through convolution, and $\omega:[0,T]\rightarrow \RR$ is a measurable path.  We will see in this section that $g$ can be chosen to be distributional given that $\omega$ is sufficiently irregular. 

Similarly as is done for pathwise regularization by noise for SDEs (e.g. \cite{Catellier2016}), we begin to consider \eqref{eq:mSHE particular} with a smooth function $g$, and  we set $\theta=u-\omega$, and then study the integral equation
\begin{equation}\label{eq: transformed eq}
    \theta_t =P_t u_0 +\int_0^t P_{t-s}\xi  g(\theta_s+\omega_s)\dd s. 
\end{equation}
The integral term can be written on the form of a non-linear Young--Volterra integral. 
To motivate this, we begin with the following observation: 
Define the linear operator  $S_{t}=P_t\xi$, whose action is defined by 
\begin{equation}\label{eq:heat volterra noise}
    S_t f:=\int_{\RR} P_t(x-y)\xi(y)f(y)\dd y. 
\end{equation}
For $\beta>\vartheta$ the classical Schauder estimates for the heat equation tells us that 
\begin{equation*}
    \|S_t f\|_{\cC^\beta} = \|S_t f\|_{\cC^{-\vartheta + 2 \frac{\beta+\vartheta}{2}}}   \lesssim t^{-\frac{\beta+\vartheta}{2}} \|\xi f\|_{\cC^{-\vartheta}}\lesssim t^{-\frac{\beta+\vartheta}{2}}\|\xi\|_{\cC^{-\vartheta}}\|f\|_{\beta}, 
\end{equation*}
where in the last estimate we have used that the product between the distribution $\xi\in \cC^{-\vartheta}$ and the function $f\in \cC^\beta$ since $\beta>\vartheta$. 
Thus we can view $S$ as a bounded linear operator from $\cC^\beta$ to itself for any $\beta>\vartheta$.
This motivates the next proposition: 

\begin{prop}\label{prop:regualrity of P as S xi}
If $\vartheta \ge 0$ and $\beta>\vartheta$, then the operator $S$ defined in \eqref{eq:heat volterra noise} is a linear operator on $\cC^\beta$ and satisfies Hypothesis \ref{kernel hypothesis} with singularity $\rho = \frac{\beta+\vartheta}{2}$.    
\end{prop}
\begin{proof}
This follows from the properties of the heat kernel proven in  Corollary \ref{cor: heat kernel estimates} in Section \ref{sub:heat_kernel}, and the fact that the para-product $\xi f$ satisfies the bound $\|\xi f\|_{\cC^{-\vartheta}}\leq\|\xi\|_{\cC^{-\vartheta}} \|f\|_{\cC^\beta}$ due to the assumption that $\beta>\vartheta$. See \cite{BahCheDan} for more details on the para-product in Besov spaces. 
\end{proof}

The integral in \eqref{eq: transformed eq} can therefore be written as 
\begin{equation*}
    \int_0^t P_{t-s} \xi g(\theta_s+\omega_s)\dd s =\int_0^t S_{t-s} g(\theta_s+\omega_s)\dd s. 
\end{equation*}
Furthermore, the classical Volterra integral on the right hand side can be written in terms of the averaged field $T^\omega g$, in the sense that 
\begin{equation*}
    \int_0^t S_{t-s} g(\theta_s+\omega_s)\dd s=\int_0^t S_{t-s} T^\omega_{\dd s}g(\theta_s),  
\end{equation*}
where for a partition $\cP$ of $[0,t]$ with infinitesimal mesh, we define
\begin{equation*}
    \int_0^t S_{t-s} T^\omega_{\dd s}g(\theta_s):=\lim_{|\cP|\rightarrow 0} \sum_{[u,v]\in \cP} S_{t-u} T^\omega_{u,v}g(\theta_u). 
\end{equation*}
It is not difficult to see (and we will rigorously prove it later) that when $g$ is smooth, the above definition agrees with the classical Riemann definition of the integral. However, the advantage with this formulation of the integral is that $T^\omega g$ might still make sense as a function when $g$ is only a distribution, and thus we will see that this formulation allows for an extension of the concept of integration to distributional coefficients $g$.

\begin{prop}\label{prop:existence of particular integral}
Consider a measurable path $\omega:[0,T]\rightarrow \RR$ and $g\in \cS'(\RR)$, and suppose  $T^\omega g\in \cC^\gamma_T\cC^\kappa(w)$ for some $\gamma>\frac{1}{2}$ and $\kappa\geq 3$. For some $0 < \beta < 1$, suppose $S:[0,T]\rightarrow \cL( \cC^\beta)$ satisfies Hypothesis \ref{kernel hypothesis} for some $0\leq \rho< 1$ such that $\gamma-\rho>1-\gamma$. Take $\gamma-\rho>\varsig >1-\gamma$ Then for any $y\in \scC^\varsig_T\cC^\beta$, the integral
\begin{equation}\label{eq:particular integral}
    \Theta(y)_t:= \int_0^t S_{t-s} T^\omega_{\dd s}g(y_s):=\lim_{|\cP|\rightarrow 0} \sum_{[u,v]\in \cP} S_{t-u} T^\omega_{u,v}g(y_u)
\end{equation}
exists as a non-linear Young-Volterra integral according to Lemma \ref{integral lemma}. 
\end{prop}

\begin{proof}
Let $E:=\cC^\beta$. Since $T^\omega g\in \cC^\gamma_T\cC^\kappa(w)$ for some $\gamma>\frac{1}{2}$ and $\kappa\geq 3$, it follows from Proposition \ref{prop: inf dim reg of avg op} that $T^\omega g$ can be seen as a function from $[0,T]\times E$ to $E$ that satisfies {\rm (i)-(ii)} in Lemma \ref{integral lemma} with $H(x):=\sup_{|z|\leq x} w^{-1}(z)\|T^\omega g\|_{\cC^\gamma_T\cC^\kappa(w)}$.  Lemma \ref{integral lemma} then implies that \eqref{eq:particular integral} exists, and satisfies the bound in \eqref{eq:bound on theta int }. 
\end{proof}

As an immediate consequence of the above proposition and Proposition \ref{prop: stability of theta}, we obtain the following Proposition: 
\begin{prop}\label{prop:particularstability}
Under the same assumptions as in Proposition \ref{prop:existence of particular integral}, for two paths $x,y\in \scC^\varsig_T\cC^\beta$ we have that 
\begin{equation*}
    [\Theta(y)-\Theta(x)]_\varsig\lesssim_{P,\xi} \Big\{\sup_{z\leq \|x\|_\infty \vee \|y\|_\infty} w^{-1}(z) \|T^\omega g\|_{\cC^\gamma_T\cC^\kappa}\Big\}\left( \|x_0-y_0\|_{\cC^\beta}+ [x-y]_\varsig\right).
\end{equation*}
Moreover, let $\{g_n\}_{n\in \NN}$ be a sequence of smooth functions converging to $g\in \cS'(\RR)$ such that $T^\omega g_n\rightarrow T^\omega g\in \cC^\gamma_T\cC^\kappa(w)$. 
Then the sequence of integral operators $\{\Theta_n\}_{n\in \NN}$ built from $T^\omega g_n$ converge to $\Theta$, in the sense that for any $y\in \scC^\varsig_T\cC^\beta$
\begin{equation*}
    \lim_{n\rightarrow \infty}\|\Theta_n(y)-\Theta(y)\|_{\cC^{\varsig}_T\cC^\beta}=0,
\end{equation*}
for any $1-\gamma<\varsig<\gamma-\rho$. 
\end{prop}

\begin{proof}
A combination of Proposition \ref{prop:existence of particular integral} and Proposition \ref{prop: stability of theta} gives the first claim. For the second, consider the field $X^n_{s,t}(x)=\left(T^\omega_{s,t} g_n-T^\omega_{s,t}g\right)(x)$. Then $X$ satisfies condition {\rm (i)-(ii)} of Lemma \ref{integral lemma} with $H(x)=\sup_{|z|\leq x} w^{-1}(z) \|T^\omega g_n-T^\omega g\|_{\cC^\gamma_T\cC^\kappa(w)}$.  Thus by Proposition \ref{prop: stability of theta} it follows that the integral operator $\Theta^{X^n}:\scC^\varsig \cC^\beta \rightarrow \scC^\varsig \cC^\beta$ built from $X^n$ satisfies for any $y\in \scC^\varsig_T\cC^\beta$
\begin{equation*}
    \|\Theta^{X^n}(y)\|_{\cC^\varsig_T\cC^\beta} \lesssim_{S,\xi}  \sup_{|z|\leq \|y\|_\infty}w^{-1}(z) \|T^\omega g_n-T^\omega g\|_{\cC^\gamma_T\cC^\kappa(w)}(1+[y]_\varsig). 
\end{equation*}
Thus taking the limits when $n\rightarrow \infty$, it follows that $\|\Theta^{X^n}(y)\|_{\cC^\varsig_T\cC^\beta}\rightarrow 0$ due to the assumption that $\|T^\omega g_n-T^\omega g\|_{\cC^\gamma_T\cC^\kappa(w)}\rightarrow 0$ when $n\rightarrow \infty$. 
\end{proof}

 We are now ready to prove existence and uniqueness of solutions to \eqref{eq: transformed eq} in the non-linear Volterra-Young formulation, through an application of the abstract existence and uniqueness results of equations on the form
\begin{equation*}
    \theta_t= p_t+\Theta(\theta)_t,\quad t\in [0,T], 
\end{equation*}
proven in Theorem \ref{thm: abs existence and uniqueness}.  In combination with Proposition \ref{prop:existence of particular integral} and \ref{prop:particularstability} we consider 
 $p_t=P_t \psi$ for some $\psi\in \cC^\beta$, $\Theta$ is the integral operator constructed in \eqref{eq:particular integral} with $S_{t}=P_{t}\xi$ for some $\xi\in \cC^{-\vartheta}$ with $\beta>\vartheta$.

\begin{thm}\label{thm: existence and uniqunes theta}
Consider a measurable path $\omega:[0,T]\rightarrow \RR$ and $g\in \cS'(\RR)$, and suppose  $T^\omega g\in \cC^\gamma_T\cC^\kappa(w)$ for some $\gamma>\frac{1}{2}$ and $\kappa\geq 3$. 

Let $0<\vartheta < 1$ and $\vartheta<\beta < 2-\vartheta$. Let $\xi\in \cC^{-\vartheta}$ and $\rho = \frac{\beta - \vartheta}{2}$ and suppose that $1-\gamma<\gamma - \rho$. Suppose $\psi\in \cC^\beta$. Then there exists a time $\tau\in (0,T]$ such that there exists a unique $\theta\in \cB_\tau(P\psi)\subset \scC^\varsig_\tau \cC^\beta$ which solves the non-linear Young equation
\begin{equation}\label{eq:thetaeq particular}
    \theta_t = P_t \psi +\int_0^t P_{t-s}\xi T^\omega_{\dd s}g(\theta_s),\quad t\in [0,\tau],
\end{equation}
for any $\gamma-\rho >\varsig > 1-\gamma$ where $S_t:=P_t\xi$ as defined Proposition \ref{prop:regualrity of P as S xi}, and the integral is understood in the sense of proposition \ref{prop:existence of particular integral}. There exists a $C=C({P,\xi})>0$ such that the solution $\theta$ satisfies the following bound 
\begin{equation}\label{eq:bound local solution particular}
       [\theta]_{\varsig;\tau} \leq  C\left([P\psi]_{\varsig,\tau} +  2\sup_{|z|\leq 1+\|\psi\|_{\cC^\beta} +[P\psi]_{\varsig;\tau}}w^{-1}(z) \|T^\omega g\|_{\cC^\gamma_T\cC^\kappa}\right).  
\end{equation}
Moreover, if $w\simeq 1$, then there exists a unique $\theta\in \scC^\varsig_T\cC^\beta$ which is a global  solution, in the sense that $\tau=T$. 
\end{thm}

\begin{proof}
By Proposition \ref{prop:regualrity of P as S xi}, it follows that $S_t:=P_t\xi$ is a linear operator on $\cC^\beta$ since $\beta>\vartheta$. It then follows from Proposition \ref{prop:existence of particular integral} that the non-linear Young integral 
\begin{equation*}
    \Theta(y)_t:=\int_0^t S_{t-s} T^\omega_{\dd s}g(y_s),
\end{equation*}
exists as a map $\Theta:[0,T]\times \scC^\varsig_T\cC^\beta \rightarrow \scC^\varsig_T\cC^\beta$, and we have that for any $y\in \scC^\varsig_T\cC^\beta$ with $\gamma-\rho>\varsig>1-\gamma$ there exists a constant $C>0$ such that
\begin{equation}\label{eq:bound theta 4}
    [\Theta(y)]_\varsig \leq C \sup_{|z|\leq \|y\|_\infty} w^{-1}(z) \|T^\omega g\|_{\cC^\gamma_T\cC^\kappa(w)} (1+[y]_\varsig)T^{\gamma-\rho-\varsig}. 
\end{equation}
Moreover by Proposition \ref{prop:particularstability} we have that for two paths $x,y\in \scC^\varsig_T\cC^\beta$ there exists a constant $C=C(S,\xi)>0$ such that
\begin{equation}\label{eq:diff theta 4}
    [\Theta(x)-\Theta(y)]_\varsig \leq C \sup_{|z|\leq \|x\|_\infty\vee \|y\|_\infty} w^{-1}(z) \|T^\omega g\|_{\cC^\gamma_T\cC^\kappa(w)}\left(\|x_0-y_0\|_{\cC^\beta}+ [x-y]_\varsig\right) T^{\gamma-\rho-\varsig}.
\end{equation}
Thus by theorem \ref{thm: abs existence and uniqueness}, for any $\tau>0$ such that 
\begin{equation*}
    \tau\leq \left[4(1+[P\psi]_\varsig) C \sup_{|z|\leq 1+ \|\psi\|_{\cC^\beta}+ [P\psi]_\varsig} w^{-1}(z) \right]^{-\frac{1}{\gamma-\rho-\varsig}},
\end{equation*}
there exists a unique $\theta\in \scC^\varsig([0,\tau];\cC^\beta(\RR))$ for $1-\gamma<\varsig<\gamma-\rho$ which satisfies \eqref{eq:thetaeq particular}. Thus local existence and uniqueness holds for \eqref{eq:thetaeq particular}.  

We now move on to prove the bound in \eqref{eq:bound local solution particular} First observe that 
\begin{equation*}
    [\theta]_{\varsig;\tau} \leq [P\psi]_{\varsig,\tau} +[\Theta (\theta)]_{\varsig;\tau}.
\end{equation*}
From Lemma \ref{integral lemma}, we know that there exists a $C=C(P,\xi)>0$ such that
\begin{equation*}
[\Theta (\theta)]_{\varsig;\tau} \leq C\sup_{|z|\leq \|\theta\|_\infty}w^{-1}(z) \|T^\omega g\|_{\cC^\gamma_T\cC^\kappa} (1+[\theta]_{\varsig;\tau})\tau^{\gamma-(\rho+\varsig)}
\end{equation*}
Recall from Theorem \ref{thm: abs existence and uniqueness} that $\tau>0$ is chosen such that $\theta$ is contained in the unit ball centred at $S\psi$, i.e.  $\theta\in \cB_\tau(P\psi)$,  and thus $[\theta]_{\varsig;\tau}\leq [P\psi]_{\varsig;\tau}+1$, which yields  
\begin{equation*}
     [\theta]_{\varsig;\tau} \leq [P\psi]_{\varsig,\tau} +  2\sup_{|z|\leq 1+\|\psi\|_{\cC^\beta} +[P\psi]_{\varsig;\tau}}w^{-1}(z) \|T^\omega g\|_{\cC^\gamma_T\cC^\kappa}. 
\end{equation*}
At last, if $w\simeq 1$, then for any $x,y\in \scC^\varsig_T\cC^\beta$, there exists an $M>0$ such that 
$$
C \sup_{|z|\leq \|x\|_\infty\vee \|y\|_\infty} w^{-1}(z)\leq M.
$$
where $C$ is the largest of the constants in \eqref{eq:bound theta 4} and \eqref{eq:diff theta 4}.  By Theorem \ref{Global Existence} it follows that a global solution exists, i.e. there exists a unique $\Theta\in \scC^\varsig_T\cC^\beta$ satisfying \eqref{eq:thetaeq particular}.
\end{proof}

With the above theorem at hand, we are ready to prove Theorem \ref{thm:main existence and uniquness with existing avg op} and Corollary \ref{cor:main stability of solutions}. 

\begin{proof}[Proof of Theorem \ref{thm:main existence and uniquness with existing avg op}] 
From Theorem \ref{thm: existence and uniqunes theta}, we know that there exists a $\tau>0$ such that there exists a unique $\theta\in \scC^\varsig_\tau\cC^\beta$ which solves \eqref{eq:thetaeq particular}. Thus we say that there exists a unique local solution $u$ in the sense of Definition \ref{def: concept of solution}. As  proven in Theorem \ref{thm: existence and uniqunes theta}, there exists a global solution if $T^\omega g\in \cC^\gamma_T\cC^\kappa (w)$ with $w\simeq 1$.  This concludes the proof of theorem \ref{thm:main existence and uniquness with existing avg op}. 
\end{proof}

\begin{proof}[Proof of Corollary \ref{cor:main stability of solutions}]
Let $\{g_n\}$ be a sequence of smooth functions converging to $g$ such that $T^\omega g_n$ converges to $T^\omega g$ in $\cC^\gamma_T\cC^\kappa$, and define the field  $Y^n_{s,t}(x)=(t-s) g_n(x+\omega_s)$. It is readily checked that $Y^n$ satisfies {\rm (i)-(ii)} of Lemma \ref{integral lemma}, and thus $\Theta^{Y^n}$ is an integral operator from $\scC^\varsig_T\cC^\beta\rightarrow \scC^\varsig_T\cC^\beta$. Furthermore, it follows from Theorem \ref{thm: abs existence and uniqueness} together with the same reasoning as in the proof of Theorem \ref{thm: existence and uniqunes theta} that there exists a unique solution $\theta^n$ to the equation 
\begin{equation}\label{eq: theta n}
    \theta^n_t=P_t \psi +\Theta^{Y^n}(\theta)_t. 
\end{equation}
In particular, $\|\theta^n\|_{\scC^\varsig_\tau\cC^\beta}<\infty$ for all $n\in \NN$. 
Note that then $\Theta^{Y^n}$ corresponds to the classical Riemann integral on $\cC^\beta$. Indeed, since $g_n$ is smooth, the integral $\Theta^{Y^n}$ is given by 
\begin{equation*}
    \Theta^{Y^n}(y)_t:=\int_0^t S_{t-s}g(y_s)\dd s, 
\end{equation*}
where we recall that $S_t=P_t\xi$. Furthermore, it is readily checked that $\Theta^{Y^n}$ agrees with the integral operator $\Theta^{T^\omega g_n}$, since $g_n$ is smooth,  $T^\omega_{s,t}g_n(x)=\int_s^t g(x+\omega_r)\dd r$ is differentiable in time. Thus note that the difference between the two approximating integrals 
\begin{equation*}
    \begin{aligned}
         \Theta^{Y^n}_{\cP}(y)_t=\sum_{[u,v]\in \cP[0,t]} S_{t-u} g_n(y_u+\omega_u)(v-u)
         \\
         \Theta^{T^\omega g_n}_{\cP}(y)_t=\sum_{[u,v]\in \cP[0,t]} S_{t-u}\int_u^v g_n(y_u+\omega_r)\dd r
    \end{aligned}
\end{equation*}
 can be estimated by 
\begin{equation*}
    \left\| \sum_{[u,v]\in \cP} S_{t-u} \int_u^v g(y_u+\omega_u)-g(y_u+\omega_r)\dd r\right\|_{\cC^\beta}\rightarrow 0
\end{equation*}
when $|\cP|\rightarrow 0$. Therefore $\Theta^{Y^n}\equiv \Theta^{T^\omega g_n}$. 
Let $\theta\in \scC^\varsig_\tau\cC^\beta$ be the solution to \eqref{eq:thetaeq particular} constructed from $T^\omega g$ with integral operator $\Theta$. We now investigate the difference between $\theta$ and $\theta^n$ as found in \eqref{eq: theta n}. It is readily checked that 
\begin{equation*}
    [\theta^n-\theta]_\varsig \leq [\Theta(\theta)-\Theta^{T^\omega g_n}(\theta_n)]_\varsig +[\Theta^{T^\omega g_n}(\theta)-\Theta^{Y^n}]_\varsig. 
\end{equation*}
As already argued, the last term on the right hand side is equal to zero, and we are therefore left to show that the first term on the right hand side converge to zero. 
Invoking Proposition \ref{prop: stability of theta} (in particular \eqref{eq:stability bound}) and following along the lines of the proof of \ref{prop:particularstability}, it is readily checked that there exists a $C=C(P,\xi)>0$ such that
\begin{equation*}
    [\theta^n-\theta]_{\varsig;\tau} \leq C \sup_{|z|\leq \|\theta\|_\infty \vee \|\theta^n\|_\infty} w^{-1}(z) (\|T^{\omega}g_n -T^\omega g\|_{\cC^\gamma_T\cC^\kappa(w)}(1+[\theta]_{\varsig;\tau}\vee [\theta^n]_{\varsig;\tau})+[\theta-\theta^n]_{\varsig;\tau})\tau^{\gamma-\rho-\varsig}
\end{equation*}
We can now choose a parameter  $\bar{\tau}>0$ sufficiently small, such that 
\begin{equation*}
    [\theta^n-\theta]_{\varsig;\bar{\tau}} \leq 2 \sup_{|z|\leq \|\theta\|_\infty \vee \|\theta^n\|_\infty} w^{-1}(z) (\|T^{\omega}g_n -T^\omega g\|_{\cC^\gamma_T\cC^\kappa(w)}(1+[\theta]_{\varsig;\tau}\vee [\theta^n]_{\varsig;\tau}). 
\end{equation*}
From \eqref{eq:bound local solution particular} we know that there exists a constant $M>0$ such that 
\[\|\theta^n\|_\infty\vee \|\theta\|_\infty\vee [\theta^n]_{\varsig;\tau}\vee [\theta]_{\varsig;\tau}\leq M,\]
and it follows that 
\begin{equation}\label{eq:bound with M of diff n}
      [\theta^n-\theta]_{\varsig;\bar{\tau}} \leq 2 \sup_{|z|\leq M} w^{-1}(z) \|T^{\omega}g_n -T^\omega g\|_{\cC^\gamma_T\cC^\kappa(w)}(1+M). 
\end{equation}
The above inequality can, by similar arguments, be proven to hold for any sub-interval $[a,a+\bar{\tau}]\subset [0,\tau]$ and thus we conclude by \cite[Exc. 4.24]{Friz2014} that there exists a constant $C=C(P,\xi,M,\bar{\tau},w)>0$ such that
\begin{equation*}
     [\theta^n-\theta]_{\varsig;\tau} \leq C \|T^{\omega}g_n -T^\omega g\|_{\cC^\gamma_T\cC^\kappa(w)}. 
\end{equation*}
Taking the limit when $n\rightarrow \infty$ it follows that $\theta^n\rightarrow \theta$ due to the assumption that $T^{\omega}g_n\rightarrow T^{\omega}g$. 
If the solution is global, i.e. $w\simeq 1$, then the same arguments as above holds, and the inequality in \eqref{eq:bound with M of diff n} can be proven to hold on any sub-interval $[a,a+\bar{\tau}]\subset [0,T]$ and then by the same arguments as above $\theta^n\rightarrow \theta$ in $\scC^\varsig_T\cC^\beta$.
\end{proof}

\subsection{The drifted multiplicative SHE}\label{sec:drifted mshe}

So far we have proven existence and uniqueness of solutions to equations on the form 
\begin{equation*}
    u_t = P_t \psi +\int_0^t P_{t-s}\xi g(u_s)\dd s+\omega_t, \quad s\in [0,T],\,\,u_0=\psi\in \cC^\beta(\RR). 
\end{equation*}
That is, we have throughout the text assumed that $\xi\in \cC^{-\vartheta}$ with $\vartheta<\beta$ is a multiplicative spatial noise. More generally, it is frequently considered a drift term in the above equation as well, such that 
\begin{equation*}
      u_t = S_t \psi +\int_0^t P_{t-s} b(u_s)\dd s+ \int_0^t P_{t-s}\xi g(u_s)\dd s+\omega_t, \quad s\in [0,T],\,\,u_0=\psi\in \cC^\beta(\RR). 
\end{equation*}
Our results extends easily to this type of drifted equation as well, as long as $T^\omega b\in \cC^\gamma_T \cC^3(w)$. Indeed, again setting $\theta=u-\omega$, we consider the non-linear Young-Volterra equation 
\begin{equation}\label{eq:drifted mshe theta formulation}
    \theta_t =P_t \psi +\int_0^t P_{t-s} T^\omega_{\dd s}b(\theta_s) +\int_0^t P_{t-s}\xi T^\omega_{\dd s}g(\theta_s). 
\end{equation}
It is straightforward to construct the non-linear Young-Volterra integral of the form $\int_{0}^t P_{t-s}T^\omega_{\dd s} b(\theta_s)$. Since there is no multiplicative noise $\xi$, the linear operator $P_t$ on $\cC^\beta$ is not singular in its time argument. An application of Lemma \ref{integral lemma} would then allow to construct $y\mapsto \Theta^{b}(y):=\int_{0}^t P_{t-s}T^\omega_{\dd s} b(y_s)$ as an operator from $\scC^\varsig_T\cC^\beta $ into itself. The second integral is constructed as before as an operator from $\scC^\varsig_T\cC^\beta $ into itself; let us denote this by $\Theta^g$. 
Then \eqref{eq:drifted mshe theta formulation} can be written as the abstract equation 
\begin{equation*}
    \theta_t =P_t\psi + \Theta(\theta)_t,
\end{equation*}
where $\Theta(y):=\Theta^b(y)+\Theta^g(y)$ for $y\in\scC^\varsig_T\cC^\beta $. An application of Theorem \ref{thm: abs existence and uniqueness} then provides local existence and uniqueness. If both $T^\omega b$ and $T^\omega g$ are contained in some unweighted Besov spaces, global existence and uniqueness also holds, as proven in \ref{Global Existence}.

\section{Averaged fields with L\'evy noise}\label{sec: Averaging operators with Levy noise}
In this section we will construct the averaged field associated to a fractional Levy process, and prove it's regularizing properties. In contrast to the works of \cite{harang2020regularity}, we will not consider the regularity of the local time associated with a process in order to obtain the regularizing effect, but estimate the regularity of the averaged field $T^\omega b$ directly based on probabilistic techniques developed in \cite{Catellier2016}. This has the benefit that it improves the regularity of the averaged field, as discussed in the beginning of section \ref{averaged fields}. 

\subsection{Linear fractional Stable motion}

In this section we introduce a class of measurable processes which have the regularization property, and allows us to deal with our equations. 

\begin{defn}
Let $\alpha\in(0,2]$. We say that $L$ is a $d$-dimensional symmetric $\alpha$-stable L\'evy process if 
\begin{itemize}
    \item[{\rm (i)}] $L_0=0$,
    \item[{\rm (ii)}] $L$ has independent and  stationary increments,
    \item[{\rm (iii)}] $L$ is right c\'adl\'ag and
    \item[{\rm (iv)}] there exists a constant such that for all $\lambda \in \RR^d$ 
\[\EE[\exp(i L_t\cdot \xi)] = e^{-c_\alpha |\xi|^\alpha t}.\]
\end{itemize}  
\end{defn}

Following \cite{Taqqu1994}, we define a fractional process with respect to a stable L\'evy process as followed: 

\begin{defn}
Let $\alpha\in(0,2]$ and $L,\tilde{L}$ be two independent $d$-dimensional $\alpha$-stable symmetric L\'evy processes. Let $H\in (0,1)\backslash\{\alpha^{-1}\}$.
For all $t\in \RR_+$, we define the $\alpha$-linear fractional stable motion ($\alpha$-LFSM) of Hurst parameter $H$ by
\[L^H_t = \int_0^t (t-v)^{H-\frac{1}{\alpha}} \dd L_t + \int_0^{+\infty} (t+v)^{H-\frac1\alpha} - v^{H - \frac{1}{\alpha}} \dd \tilde{L}_{v}.\]
We extend this definition to $H=\frac{1}{\alpha}$ by setting $L^{\frac{1}{\alpha}} = L$.
\end{defn}

The existence of the LFSM is proved in \cite{Taqqu1994}, Chapter 3. Note that when $\alpha=2$, we recover the standard fractional Brownian motion (up to a multiplicative constant), and its different representations. Here we have chosen the moving average representation, in the fractional Brownian motion case, one could also use the harmonizable representation or something else. Note that when $\alpha<2$, those different representations are no longer equivalent. 

We gather here some properties about the LFSM which will become useful in the subsequent analysis of the regularity of averaged fields associated to LFSM. One can again consult \cite{Taqqu1994}, but also \cite{lifshitsSmallDeviationsFractional2005} for the first three points. The last point is a direct consequence of the definition of $L^H$.

\begin{prop}\label{prop:LFSM}
Let $\alpha \in (0,2]$ and let $H\in (0,1) \cup \{\alpha^{-1}\}$, and let $L^H$ be an $\alpha$-LFSM.
\begin{itemize}
    \item[{\rm (i)}] $L^H$ is almost surely measurable.
    \item[{\rm (ii)}] $L^H$ is continuous if and only if $\alpha =2$ or $H>\frac1\alpha$.
    \item[{\rm (iii)}] $L^H$ is $H$ self-similar and its increments are stationary.
    \item[{\rm (iv)}] For all $t\ge 0$ define  $\cF_t = \sigma\left(\{\tilde{L}_{(\tilde{ t}})_{\tilde{t} \in \RR_+}\} \cup \{L_s \, : \, s \leq t\}\right)$. Then $L^H$ is $(\cF_t)_{t\ge 0}$ adapted.
    \item[{\rm (v)}] For all $0 \leq s \leq r$, we have
    \[L^H_r = L^{1,H}_{s,r} + L^{2,H}_{s,r},\]
    with 
    \[L^{1,H}_{s,r} = \int_s^r (r-v)^{H - \frac{1}{\alpha}} \dd L_v\]
    and
    \[L^{2,H}_{s,r} = \int_0^s (r-v)^{H - \frac{1}{\alpha}} \dd L_v + \int_0^{+\infty} (r+v)^{H - \frac{1}{\alpha}} - v^{H-\frac{1}{\alpha}} \dd \tilde{L}_v.\]
    Hence, $L^{1,H}_{s,r}$ is independent of $\cF_r$ while $L^{2,H}_{s,r}$ is measurable with respect to $\cF_r$, and for all $x\in \RR^d$
    \[\EE[e^{i \xi\cdot L^{1,H}_{s,r}}] = e^{-c_\alpha |\xi|^\alpha (s-r)^{\alpha H} }.\]
\end{itemize}
\end{prop}

Note that we can (and we will) reformulate the last point by saying that for all $g\in\cS$,
\[\EE [g(x + L^{1,H}_{s,r}) ]= P^{\frac{\alpha}{2}}_{c_\alpha(r-s)^{\alpha H}}g(x)\]

Finally, let us recall a useful result about martingales in Lebesgue spaces due to Pinelis \cite{pinelisOptimumBoundsDistributions1994} and adapted in weighted spaces (see Appendix \ref{sub:weighted_lebesgue}).  

\begin{prop}\label{prop:azuma}
Let $p\ge 2$, let $w$ be an admissible weight, and let $(M_n)_n$ be a $(\Omega,\cF,(\cF_n)_n,\PP)$ martingale with value in $L^p(\RR^d,w;\RR)$. Suppose that there exists a constant $c>0$ such that for all $n\ge 0$, $\|M_{n+1} - M_n\|_{L^p(w)} \leq c$. Then there exists two constants $C_1>0$ and $C_2>0$ such that for all $x\ge 0$,
\[\PP\big(\|M_N - \EE[M_N]\|_{L^P(w)} \ge x \big) \leq C_1 e^{-\frac{C_2 x^2}{Nc^2}}. \]
\end{prop}

\subsection{Averaging operator}

In this section we give another proof of the regularizing effect of the fractional Brownian motion. This proof is similar to the one in \cite{Catellier2016}, but provides bounds in general weighted Besov spaces directly.  For other proofs of the same kind of results, one can consult \cite{coutinItoTanakaTrickNonsemimartingale2019}, \cite{galeati2020noiseless} and \cite{leStochasticSewingLemma2020}. Other proofs may rely on Burkolder-Davis-Gundy inquality in UMD spaces. For information about UMD spaces, one can consult \cite{hytonenAnalysisBanachSpaces2016}.

\begin{lem}\label{lem:average_blocks}
    Let $w$ be an admissible weight. Let $\alpha \in (0,2]$ and let $H \in (0,1)\cup \{\alpha^{-1}\}$. Let $\nu\in(0,1)$. Then there exists three constants $K,C_1,C_2>0$ such that for all  $j \ge -1$, all $g\in \cS'(\RR^d;\RR)$ all $s\leq t$ and all $x > K$
\[\PP\left( \frac{\|\int_s^t \Delta_j g(\cdot + L^H_r)\dd r\|_{L^p(w)}}{|t-s|^{1-\frac\nu2}2^{-j\frac{\nu}{2H}}\|\Delta_j g\|_{L^p(w)}}\ge x\right) \leq C_1 e^{-C_2 x^2}.\]
\end{lem}

\begin{proof}
Let $0 \leq s\leq t $ and let $N\in \NN$ to be fixed from now. Let us define $t_n = \frac{n}{N}(t-s) + s$, and $\cG_n = \cF_{t_n}$. Let us define 
\[M_n(x) = \EE\left[\int_s^t \Delta_j g (x + L^H_r)\dd r\bigg|\cF_{t_n}\right].\]
Hence, $(M_n)_n$ is a martingale with respect to $(\cG_n)_n$ and with value in $L^p(\RR^d,w;\RR)$ for all $p\ge 2$, and furthermore. Note also that 
\[M_N = \int_s^t \Delta_j g(\cdot + L^H_r)\dd r .\]
Furthermore, we have for all $0\leq n\leq N-1$
\[M_{n+1} - M_n = \int_{t_n}^{t_{n+1}} \Delta_j g(\cdot + L^H_r) \dd r + A_{n+1} - A_n,\]
with
\[A_n = \EE \left[\int_{t_n}^t \Delta_j g(\cdot + L^H_r)\dd r \bigg| \cF_{t_n}\right].\]
It is also readily checked that for all $n=0,\ldots,N$ we have 
\[\left\|\int_{t_n}^{t_{n+1}} \Delta_j g(\cdot + L^H_r) \dd r\right\|_{L^p(w)} \leq \frac{|t-s|}{N} \|\Delta_j g\|_{L^p(w)}.\]
Furthermore, using the decomposition of Proposition \ref{prop:LFSM}, one get for all $n\in \{0,\cdots,N\}$,
\begin{align*}
    A_n(x) 
    =& 
    \int_{t_n}^t \EE\left[\Delta_j g \left(x+L^{1,H}_{t_n,r}+L^{2,H}_{t_n,r}\right)\bigg|\cF_{t_n}\right] \dd r \\
    =& \int_{t_n}^{t} P^{\frac{\alpha}{2}}_{(r-t_n)^{\alpha H}} \Delta_j g (x+ L^{2,H}_{t_n,r})\dd r,
\end{align*}
Hence, thanks to Proposition \ref{prop:heat_kernel_estimates}, the following bound holds:
\[\|A_n\|_{L^p(w)} \lesssim \int_{t_n}^t e^{-c 2^{\alpha j} (r-t_n)^{\alpha H}} \|\Delta_j g\|_{L^p(w)} \lesssim 2^{-\frac{j}{H}} \|\Delta_j g\|_{L^p(w)} \int_0^{+\infty} e^{-c r^{\alpha H}}\dd r.\]
Finally, we have the bound 
\begin{equation}\|M_{n+1} - M_n\|_{L^p(w)} \lesssim \left(\frac{t-s}{N} + 2^{-\frac{j}{H}}\right)\|\Delta_j g\|_{L^p(w)}.\end{equation}
Note also that we have the straightforward bound $\|M_{n+1}- M_n\|_{L^p(w)} \lesssim (t-s)\|\Delta_j g\|_{L^p(w)}$. Therefore, by interpolation, this gives
\begin{equation}\label{eq:bound_BDG}
    \|M_{n+1} - M_n\|_{L^p(w)} \lesssim (t-s)^{1-\nu} \left(\frac{t-s}{N} + 2^{-\frac{j}{H}}\right)^\nu\|\Delta_j g\|_{L^p(w)}.
\end{equation}
It is readily checked that  
\begin{multline*}
    \PP\left(\left\|M_N \right\|_{L^p(w)} \ge x |t-s|^{1-\frac\nu2} 2^{-j\frac{\nu}{2H}} \|\Delta_j g\|_{L^p(w)}\right) \\
     \leq 
     \PP\left(\left\|M_N - \EE[M_N]\right\|_{L^p(w)} \ge \frac{x}{2} |t-s|^{1-\frac\nu2} 2^{-j\frac{\nu}{2H}} \|\Delta_j g\|_{L^p(w)}\right)
     \\
     +\PP\left(\left\|\EE[M_N]\right\|_{L^p(w)} \ge \frac{x}{2} |t-s|^{1-\frac\nu2} 2^{-j\frac{\nu}{2H}} \|g\|_{L^p(w)}\right).
\end{multline*}
Furthermore, note that 
$\EE[M_N]= \EE[A_0]$, and $\|A_0\|_{L^p(w)} \leq (t-s)\| \Delta_j g\|_{L^p(w)}$, where $\|A_0\|_{L^p(w)} \lesssim 2^{-\frac{j}{H}} \|\Delta_j  g\|_{L^p(w)}$.
Hence, there  exists a constant $K>0$ such that
\[\|\EE[M_N]\|_{L^p(w)} \leq \frac{K}{2} |t-s|^{1-\frac\nu2}2^{-j\frac{\nu}{2H}}\|\Delta_j g\|_{L^p(w)}.\]
Therefore, for $x>K$,
\[\PP\left(\left\|\EE[M_N]\right\|_{L^p(w)} \ge \frac{x}{2} |t-s|^{1-\frac\nu2} 2^{-j\frac{\nu}{2H}} \|\Delta_j g\|_{L^p(w)}\right) = 0.\]
Finally, this gives us, thanks to Proposition \ref{prop:azuma}, there exists a constant $\tilde{C}_2$ such that
\begin{multline*}
    \PP\left(\left\|\int_s^t \Delta_j g(\cdot + L^H_r)\dd r \right\|_{L^p(w)} \ge x |t-s|^{1-\frac\nu2} 2^{-j\frac{\nu}{2H}} \|\Delta_j g\|_{L^p(w)}\right) 
    \\
    \lesssim \exp\left(-\tilde{C}_2 \frac{x^2|t-s|^{2-\nu} 2^{-j\frac{\nu}{H}}\|\Delta_j g\|^2_{L^p(w)}}{N\left((t-s)^{1-\nu}\left(\frac{t-s}{N} + 2^{-\frac{j}{H}}\right)^\nu\right)^2\|\Delta_j g\|^2_{L^p(w)}}\right).
    \end{multline*}
When optimizing in $N$, one obtains the desired result.
\end{proof}

\begin{rem}
Thanks to the bound \eqref{eq:bound_BDG} and to UMD space standard argument (see for example \cite{hytonenAnalysisBanachSpaces2016}, one could rather use the Burkoldher-Davis-Gundy inequality, and get \emph{for all $1<p<+\infty$}, and all $m\ge2$
\[\EE\left[\left\|\int_s^t \Delta_j g(\cdot + L^H_r)\dd r\right\|^m_{L^p(w)}\right] \lesssim_m |t-s|^{\frac{m}{2}} 2^{-\frac{m}{2H}j}\|\Delta_j g\|^m_{L^p(w)}.\]
This would allow for a more standard proof of regularizing effect  of LFSM (using Kolmogorov continuity theorem), but prevents to have Gaussian tails. 
\end{rem}

\begin{lem}\label{lem:moments}
Let $X$ be a non-negative random variable such that there exists $K,C_1,C_2>0$ such that for all $x>K$,
\[\PP(X \ge x) \leq C_1 e^{-C_2 x^2}.\]
Then for all $m \ge 1$
\[\EE[X^{2m}] \leq K^{2m} + \frac{C_1}{2} \frac{m!}{C_2^m}.\]
\end{lem}

\begin{proof}
A simple computation reveals that 
\begin{align*}
    \EE[X^{2m}] & = \int_0^{K} 2m X^{2m-1} \PP(X \ge x) \dd x + \int_K^{+\infty} 2m X^{2m-1} \PP(X\ge x) \dd x \\
    & \leq  K^{2m} + C_1\int_0^{+\infty} e^{-C_2 x^2} \dd x\\
    & = K^{2m} + C_1 C_2^{-m}\int_0^{+\infty} e^{-x^2} \dd x\\
    &= K^{2m}+\frac{C_1}{2} \frac{m!}{C_2^m}.
\end{align*}
\end{proof}

Finally, we have all the tools to prove the following theorem of regularity of the averaging operator with respect the LFSM. We will use a version of the Garsia-Rodemich-Rumsey inequality from Friz and Victoir (\cite{Friz2009} -Theorem 1.1 p. 571).

\begin{thm}\label{thm:averaging field}
Let $w$ be an admissible weight. Let $\alpha \in(0,2]$. Let $H\in(0,1)\cup\{\alpha^{-1}\}$. Let $\kappa \in \RR$ and $\nu\in[0,1]$. Let $2\leq p \leq \infty$ and let $1 \leq q\leq +\infty$. 
There exists a constant $C>0$ such that for any $g\in B^{\kappa}_{p,q}(w)$ a positive random variable $F$ exists with
\[\EE[e^{C F^2}] <+\infty\],
and such that for any $0 \leq s \leq t \leq T$
\[\|T^{L^H}_{s,t} g\|_{B^{\kappa+ \frac{\nu}{2H}}_{p,q}(w)} \leq F |t-s|^{1-\frac{\nu}{2} - \frac{\eps}{2}} \|g\|_{B^\kappa_{p,q}(w)} \quad \text{if} \quad q < + \infty\]
and
\[\|T^{L^H}_{s,t} g\|_{B^{\kappa + \frac{\nu}{2H} - \delta}_{p,\infty}(w)} \leq F |t-s|^{1-\frac{\nu}{2} - \frac{\eps}{2}} \|g\|_{B^\kappa_{p,\infty}(w)}.\]
\end{thm}

\begin{proof}
First consider $1\leq q <  \infty$. Thanks to Lemma \ref{lem:average_blocks} and Lemma \ref{lem:moments}, we know that there exists a constant $c_q >0$ such that for all $0 \leq s<t\leq T$
\[\EE\left[\left(\frac{\left\|\Delta_j T^{L^H}_{s,t} g\right\|_{L^p(w)}}{|t-s|^{1-\frac{\nu}{2}} 2^{-j \frac{\nu}{2H}}\left\|\Delta_j g\right\|_{L^p(w)}}\right)^q\right] \leq c_q\]
and that for all $m\ge 1$,
\[\EE\left[\left\|\Delta_j T^{L^H}_{s,t} g\right\|^{2m}_{L^p(w)}\right]
\lesssim\left(
K^{2m} + \frac{m!}{C_2^m}
\right)|t-s|^{(2 - \nu)m}2^{-j \frac{\nu m}{H}}\left\|\Delta_j g\right\|^{2m}_{L^p(w)}.\]
Let us take $C<C_2$. We have 
\[\EE\left[\exp\left(C \left(\frac{\|T^{L^H}_{s,t} g\|_{B^{\kappa + \frac{\nu}{2H}}_{p,q}(w)} }{|t-s|^{1-\frac{\nu}{2}} \|g\|_{B^\kappa_{p,q}(w)}}\right)^2\right)\right] = A + B,\]
with
\[A=
\sum_{2m \leq q} \frac{C^m}{m!|t-s|^{(2-\nu)m}\|g\|^{2m}_{B^\kappa_{p,q}(w)}} \EE\left[\left(\sum_{j\ge 1} 2^{j\left(\kappa + \frac{\nu}{2H}\right)q} \left\|\Delta_j T^{L^H}_{s,t} g\right\|^{q}_{L^p(w)}\|\right)^{\frac{2m}{q}}\right]\]
and
\[ B =
\sum_{2m>q} \frac{C^m}{m!} 
\EE\left[
    \left(
        \sum_{j\ge{-1}}
            \frac{2^{jq \kappa} \|\Delta _j g\|^q_{L^P(w)}} {\sum_{j\ge{-1}}2^{jq \kappa} \|\Delta _j g\|^q_{L^P(w)}} 
    \left(\frac{\|\Delta_j T^{L^{H}}_{s,t}g\|_{L^p(w)}}{|t-s|^{1-\frac\nu2} 2^{-j\frac{\nu}{2}}\|\Delta _j g\|_{L^P(w)}}\right)^q
    \right)^{\frac{2m}{q}}
\right].
\]
For $A$, we use Jensen inequality in the concave case, and we have thanks to Lemma \ref{lem:moments}
\begin{align*}
    A \leq & \sum_{2m \leq q} \frac{C^m}{m!|t-s|^{(2-\nu)m}\|g\|^{2m}_{B^\kappa_{p,q}(w)}} \left(\sum_{j\ge 1} 2^{j\left(\kappa + \frac{\nu}{2H}\right)q} \EE\left[ \left\|\Delta_j T^{L^H}_{s,t} g\right\|^{q}_{L^p(w)}\|\right]\right)^{\frac{2m}{q}}\\
    \lesssim &
    \sum_{2m < q} \frac{C^m \left(c_q^\frac{2}{q}\right)^m}{m!} \\
    \lesssim & e^{C c_q^\frac{2}{q}}.
\end{align*}
For $B$, we use Jensen inequality in the convex case, and we have again thanks to Lemma \ref{lem:moments}
\begin{align*}
    B = &
\sum_{2m>q} \frac{C^m}{m!} 
\EE\left[
    \left(
        \sum_{j\ge{-1}}
            \frac{2^{jq \kappa} \|\Delta _j g\|^q_{L^P(w)}} {\sum_{j\ge{-1}}2^{jq \kappa} \|\Delta _j g\|^q_{L^P(w)}} 
    \left(\frac{\|\Delta_j T^{L^{H}}_{s,t}g\|_{L^p(w)}}{|t-s|^{1-\frac\nu2} 2^{-j\frac{\nu}{2}}\|\Delta _j g\|_{L^P(w)}}\right)^q
    \right)^{\frac{2m}{q}}
\right] \\
\leq &
\sum_{2m>q} \frac{C^m}{m!} 
    \sum_{j\ge{-1}}
            \frac{2^{jq \kappa} \|\Delta _j g\|^q_{L^P(w)}} {\sum_{j\ge{-1}}2^{jq \kappa} \|\Delta _j g\|^q_{L^P(w)}} 
    \EE\left[\left(\frac{\|\Delta_j T^{L^{H}}_{s,t}g\|_{L^p(w)}}{|t-s|^{1-\frac\nu2} 2^{-j\frac{\nu}{2}}\|\Delta _j g\|_{L^P(w)}}\right)^{2m}\right]\\
    \lesssim &
    \sum_{2m>q} \frac{C^m}{m!} \left(K^{2m} + \frac{m!}{C_2^m}\right)\\
    \lesssim &
    e^{C K^2} + \frac{1}{1-\frac{C}{C_2}}.
\end{align*}
Hence, if we define 
\[F= \int_0^T \int_0^T \exp\left(C \left(\frac{\|T^{L^H}_{s,t} g\|_{B^{\kappa + \frac{\nu}{2H}}_{p,q}(w)} }{|t-s|^{1-\frac{\nu}{2}} \|g\|_{B^\kappa_{p,q}(w)}}\right)^2\right) \dd s \dd t,\]
we are exactly in the scope of the Garsia-Rodemich-Rumsey inequality with $\psi (x) = e^{C x^2}$ and $p(u) = u^{1-\frac{\nu}{2}}$, which leads to the wanted result when $q<+\infty$.
When $q = \infty$ one needs to deal with the supremum in $q$, and thus needs to lose a bit in $q$. We leave it to the reader since its a direct adaptation of the previous proof. 
\end{proof}

With the above construction of the averaged field associated with the fractional L\'evy process, we are ready to prove Theorem \ref{thm:main fractional Levy}. 

\begin{proof}[Proof of Theorem \ref{thm:main fractional Levy}]
As in Section \ref{sec: existence and uniqueness of mSHE}, let us first consider a distribution $\xi \in \cC^{-\vartheta}$ and an initial condition $\psi \in \cC^{\beta}$ with $0<\vartheta<1$ and $\beta = \vartheta + \eps_1$.
The singularity $\rho = \frac{\vartheta + \beta}{2}$ is one of the limiting thresholds in all the previous computations. Hence, the best choice of $2-\vartheta>\beta>\vartheta$ for the space of the initial condition is  for a small $\eps_1>0$. 
Note also that in Theorem \ref{thm:averaging field} one want to take $\nu$ as close as possible to 1, since $\frac{\nu}{2H}$ is somehow the index of regularization of the LFSM. Note also that for some $\eps_2>0$ small enough, one has
\[\gamma = 1 - \frac{\nu}{2} - \eps_2.\]
The condition which allows the nonlinear Young--Volterra calculus to work is 
\[\frac{1}{2}< \gamma-\rho \]
here it gives us
\[\vartheta < 1 - \nu - \frac{\eps_1}{2} - \eps_2,\]
which gives for some $\eps_3>0$
\[\nu = 1 - \vartheta -  \frac{\eps_1}{2} - \eps_2 - \eps_3.\]
Hence, for an admissible weight $w$ and a function $g\in \cC^{\kappa}$ with $\kappa > 3 - \frac{1-\vartheta}{2H}$, there exists $\eps_1>0$, $\eps_2>0$ and $\eps_3>0$ such that all the previous condition are satisfied and such that almost surely
\[T^{L^H} \in \cC^{\gamma}_T \cC^{3}(w).\]

Applying Theorem \ref{thm:main existence and uniquness with existing avg op} and \ref{cor:main stability of solutions} we have the desired result.
\end{proof}

\begin{proof}[Proof of Corollary \ref{cor:white noise}]
The proof is straightforward when one recall that the space white noise $\xi$ is almost surely in $\cC^{-\vartheta}$ for any $\vartheta > \frac{d}{2}$. One can then apply Theorem \ref{thm:main fractional Levy}.
\end{proof}

\section{Conclusion}\label{sec:conclussion}

We have proven that  a certain class of measurable paths $\omega:[0,T]\rightarrow \RR$ provides strong regularizing effects on the multiplicative stochastic heat equation of the form of \eqref{eq:formalequationdef}. In particular, we prove that there exists measurable paths $\omega$ such that  local existence and uniqueness of such equations holds even when the non-linear function $g$ is only a Schwartz distribution and $\xi\in \cC^{-\vartheta}$ for some $\vartheta <1$, thus allowing for very rough spatial noise in the one dimensional setting. To this end, we apply the concept of non-linear Young integration, and extends this to the infinite dimensional setting with  Volterra operators.   This sheds new light on the application of the "pathwise regularization by noise" techniques developed in \cite{Catellier2016} to the context of SPDEs, and we believe that this program can be taken further in several directions in the future, and we provide some thoughts on such developments here. 

In the current article we restricted our analysis to the spatial space white noise on $\TT$ (Corollary \ref{cor:white noise}). Our techniques could be extended to $\TT^d$, but with the same techniques, one could not allow for white spatial noise $\xi$. Indeed, recall that if $\xi:\TT^d\rightarrow \RR$ then $\xi\in \cC^{-\frac{d}{2}-\eps}$ for any $\eps>0$ and thus even in $d=2$ the noise is too rough, since the best we can deal with in the Young setting developed above is when $\xi\in \cC^{-\vartheta}$ with $\vartheta<1$. However, there is a possibility that techniques from the theory of paracontrolled calculus as developed in \cite{Gubinelli2015Paracontrolled} could be applied here to make sense of the product, and thus a generalization could then be possible. In that connection, one would possibly need "second order correction terms" associated to the averaged field $T^\omega g$ which at this point is unclear (at least to us) how one should construct. 

Another possible direction would be to allow multiplicative space-time noise, i.e. consider $\xi$ as a distribution on $[0,T]\times \RR^d$ (or $\TT^d$). When $\xi$ is depending on time, one can no longer use the non-linear Young integral in the same way as developed here. The situation looks like the one encountered when considering regularization by noise for ODEs with multiplicative (time dependent) noise of the form 
\begin{equation}\label{eq:}
    y_t=y_0+\int_0^t b(y_s+\omega_s)\dd \beta_s+\omega_t, \qquad y_0\in \RR^d. 
\end{equation}
Existence and uniqueness of the above equation when $\beta$ is a fractional Brownian motion with $H>\frac{1}{2}$ was recently established in \cite{galeati2020regularization}, even when $b$ is a distribution. The key idea in this result was to consider the average operator 
\begin{equation*}
    \Gamma_{s,t}^\omega b(x):=\int_s^t b(x+\omega_r)\dd \beta_r
\end{equation*}
and use a recently developed probabilistic lemma by Hairer and Li \cite{hairer2020} to show that the regularity of $\Gamma^\omega b$ is linked to the regularity of $T^\omega b$. 
By considering an averaging operator given on the form on a Banach space
\begin{equation}
    \Pi_{s,t}^\omega b(x)=\int_s^t b(x+\omega_s)\dd \xi_s, 
\end{equation}
where $x\in E$ and $\xi_t\in E$ is a time-colored spatially-white noise, for some Banach space $E$. If one can extend the lemmas of Hairer and Li to the infinite dimensional setting, showing the connection between the regularity of $\Pi^\omega b$ and $T^\omega b$ (when considering $T^\omega b$ as an infinite dimensional averaged field, as in Proposition \ref{prop: inf dim reg of avg op}) there is a possibility that one could prove regularization by noise for stochastic heat equations with space-time noise on the form 
\begin{equation*}
    x_t =P_t\psi +\int_0^t P_{t-s}g(x_s)\dd \xi_s+\omega_t,  
\end{equation*}
by using similar techniques as developed in the current article. 
We leave a deeper investigation into these possibilities open for future work.

\appendix

\section{Basic concepts of Besov spaces and properties of the heat kernel}

We gather here some material about Besov spaces, heat kernel estimates and embedding in those spaces. This section is strongly inspired by  \cite{BahCheDan} and \cite{Triebel2006}. See also \cite{Mourrat2017Global} for the weighted besov norms. For the sake of the comprehensions, we give elementary proofs of the main points for developing the theory. This is the purpose of subsection \ref{sub:weighted_lebesgue} and \ref{sub:weighted_besov}.
For the sake of the Volterra sewing lemma (see Section \ref{sec:sewing}), we need a few non-standard estimates on the heat kernel action on Besov spaces. This is the purpose of subsection \ref{sub:heat_kernel}. In Section \ref{sec:CL-standard} we prove the elementary Cauchy-Lipschitz theorem for multiplicative SHE without additive perturbation. 

\subsection{Weighted Lebesgue spaces}\label{sub:weighted_lebesgue}

In order to work in a general setting, we will define the weighted Besov spaces. To do so, let us define the class of admissible weight, following Triebel chapter 6 \cite{Triebel2006}. In \cite{Mourrat2017Global} one can find a more general definition for weights, which somehow allows the same kind of estimates. 

\begin{defn}
We say that $w\in C^{\infty}(\RR^d;\RR_+\backslash\{0\})$ is an admissible weight function if 
\begin{itemize}
\setlength\itemsep{0.5em}
    \item[{\rm (i)}] 
    For all $\kappa \in \NN^{d}$, there exists a positive constant $c_\kappa$ such that 
    \[\forall x\in\RR^d,\quad \partial^\kappa w(x) \leq c_\kappa w(x).\]
    \item[{\rm (ii)}]
    There exists $\lambda'\ge 0$ and $c>0$ such that
    \begin{equation}\label{eq:lambda'}
    w(x) \leq c w(y) (1+|x-y|^2)^{\frac{\lambda'}{2}}.
    \end{equation}
    \end{itemize}
    Furthermore for any admissible weight we define the weighted $L^p$ spaces as the following :
    \[L^p(\RR^d;\RR|w) = \{f:\RR^d \to \RR \, : \, \|f\|_{L^p(\RR^d,w)} = \|wf\|_{L^p(\RR^d;\RR)}<+\infty\}.\]
\end{defn}

A direct consequence of the definition is that the product of two admissible weights is also an admissible weight. Furthermore, standard polynomials weights are of course admissible, as proved in the following Proposition. Finally, H\"older inequality in weighted Lebesgue spaces is straightforward. 

\begin{prop}
Let $\lambda \in \RR$, and let us define for all $x\in \RR^d$, $\langle x \rangle = (1+|x|^2)^{\frac{1}{2}}$. Then $\langle \cdot \rangle^\lambda$ is an admissible weight with $\lambda' = |\lambda|$.
\end{prop} 

\begin{proof}
Let us remark that for all $x,y\in\RR^d$
\[(1+y\cdot(x-y) )^2 \leq 2 \left(1+|y|^2 |x-y|^2\right) \leq 2(1+|y|^2|x-y|^2)^2 .\]
Hence, 
\[1+y\cdot(x-y) \leq \sqrt{2} (1+|y|^2|x-y|^2),\]
then
\begin{align*}
\la x \ra^2 
= & 
1 + |x-y|^2 + |y|^2 + 2y \cdot (x-y) \\
\leq & |x-y|^2 + |y|^2 + 2\sqrt{2}\left(1+|y|^2|x-y|^2\right)\\ 
\leq &  2\sqrt{2}\left(1+|x-y|^2 + |y|^2 +|y|^2|x-y|^2\right)\\
= &2\sqrt{2}\la y\ra^2 \la x-y \ra^2.
\end{align*}
and it therefore follows that
\[\langle x \rangle \leq 2^{\frac34} \langle y\rangle \langle x-y\rangle.\]
 If $\lambda \ge 0$, 
\[\langle x \rangle^\lambda \leq 2^{\frac{\lambda 3 }{4}}\langle y\rangle^\lambda \langle x-y\rangle^\lambda,\]
and 
\[\langle x \rangle^{-\lambda} =\langle x-y \rangle^{\lambda}\left(\langle x \rangle \langle y-x \rangle \right)^{-\lambda} \leq 2^{\frac{3\lambda }{4}}\langle y\rangle^{-\lambda}\langle x-y \rangle^{\lambda}, \]
which concludes the proof.
\end{proof}

\begin{lem}
Let $w$ be an admissible weight and $\lambda'$ defined as in Equation \eqref{eq:lambda'}. Then for $1\leq p,q,r \leq +\infty$ with $\frac{1}{p} + \frac{1}{q} = 1 + \frac{1}{r}$ and any measurable functions $f$ and $g$,
\[\|f*g\|_{L^r(w)} \leq \|f\|_{L^p(\langle\cdot\rangle^{\lambda'})} \|g\|_{L^p(w)}.\]
\end{lem}

\begin{proof}
The result is a direct consequence of the definition of admissible weights and of standard Young inequality. Indeed,
\begin{align*}
    \|f*g\|_{L^r(w)}
        =&
        \left(\int_{\RR^d} \left|\int_{\RR^d} f(y) g(x-y)\dd y\right|^r w(x)^r\dd x\right)^{\frac{1}{r}} \\
        \lesssim &
        \left(\int_{\RR^d} \left|\int_{\RR^d} \la x-y\ra^{\lambda'}f(x-y) w(y) g(y)\dd y\right|^r w(x)^r\dd x\right)^{\frac{1}{r}} \\
        =&
        \|(\llrr^{\lambda'}f)*(wg)\|_{L^r}\\
        \leq &
        \|\llrr^{\lambda'}f\|_{L^p}\|wg\|_{L^q}\\
        =&
        \|f\|_{L^p(\llrr^{\lambda'})}\|g\|_{L^q(w)}
\end{align*}
which is the desired result. 
\end{proof}

Furthermore, as usual in order to define Besov spaces we will work with functions with compactly supported Fourier transform. In order to deal with such functions, let us prove a Bernstein-type lemma in weighted $L^p$ space:

\begin{lem}\label{lem:bernstein}
Let $w$ be an admissible weight. Let $\cC = \{\xi\in \RR^d\, : \, c_1\leq |\xi| \leq c_2\}$ be an annulus and $\cB$ be a ball. There exists a constant $C>0$ such that for all $1\leq p \leq p' \leq +\infty$, $n\ge 0$,  $a\ge 1$,  and for any function $f \in L^p(w)$, we have
\begin{itemize}
\setlength\itemsep{0.5em}
    \item[{\rm (i)}] If $\mathrm{supp} \hat{f}  \subset a \cB$ then for any $k\in\NN^d$
    \[\|D^n f\| = \sup_{|k|=n}\|\partial^k f\|_{L^{p'}(w)} \leq C^{n +1} a^{n + d \left(\frac{1}{p}-\frac{1}{p'}\right)} \|f\|_{L^p(w)}.\]
    
    \item[{\rm (ii)}] If $\mathrm{supp} \hat{f}  \subset a \cA$  then 
    \[\frac{1}{C^{n+1}} a^{n} \|f\|_{L^P(w)} \leq \|D^n f\|_{L^p(w)} \leq C^{n+1} a^{n} \|f\|_{L^p(w)}.\]
\end{itemize}
Here $\hat{f}$ is the Fourier transform of $f$.
\end{lem}

\begin{proof}
Let $\hat{K}$ be a function such that $\hat K \equiv 1$ on $\cB$ and  $\mathrm{supp}(\hat{K})$ is compactly supported and let us define $K_a = a^{d} K(a\cdot) $. If $\mathrm{supp} \hat{f} \subset a \cB$, then $f = K_a * f$, and $\partial^k f = (\partial^k K_a )* f = a^{|k|} ((\partial^k K)_a )* f$, where $(\partial^k K)_a = a^d \partial^k K (a\cdot) $. Hence, by the previous weighted Young inequality, for $1\leq p \leq p'\leq +\infty$ and $\frac{1}{r} = 1 -\left(\frac{1}{p} - \frac{1}{p'}\right)$,
\begin{equation*}
\|\partial^k f\|_{L^{p'}(w)} 
    =
    a^{|k|}\|(\partial^k K)_a * f\|_{L^{p'}(w)}\\
    \le
    a^{|k|} \|(\partial^k K)_a\|_{L^r(\llrr^{\lambda'})}  \|f\|_{L^{p}(w)}.
\end{equation*}
Furthermore, since $a\ge 1$ and $\lambda'\ge 0$, one has $\langle a^{-1}x \rangle^{\lambda'} \leq \langle x \rangle^{\lambda'}$, and
\begin{align*}
    \|(\partial^k K)_a\|_{L^r(\llrr^{\lambda'})}^r 
        = &
        a^{rd} \int_{\RR^d} |\partial^k K(ax)|^r \la x \ra^{r\lambda'} \dd x \\
        = & a^{(r-1)d} \int_{\RR^d} |\partial^k K(x)|^r \la a^{-1}x \ra^{r\lambda'} \dd x \\
        \leq &
        a^{(r-1)d} \int_{\RR^d} |\partial^k K(x)|^r \la x \ra^{r\lambda'} \dd x.
\end{align*}
It follows that
\[\|(\partial^k K)_a\|_{L^r(\llrr^{\lambda'})} \lesssim a^{\left(1-\frac{1}{r}\right)} \lesssim a^{d\left(\frac{1}{p}-\frac{1}{p'}\right)}.\]
Gathering the above considerations proves {\rm (i)}. 

The second inequality of {\rm (ii)} is just a sub-case of {\rm (i)}. For the first inequality of the {\rm (ii)},  consider a smooth function $L$ such that $\mathrm{supp} \hat{L}$ is included in an annulus and such that $\hat{L} \equiv 1$ on $\cA$. Following \cite{BahCheDan} Lemma 2.1 and (1.23) page 25, there exists some real numbers $(A_k)$ such that 
\[|\xi|^{2n} = \sum_{|k|=n} A_k (-i\xi)^k (i\xi)^k,\]
with $(\xi_1,\cdots,\xi_d)^{(k_1,\cdots,k_d)} = \xi_1^{k_1} \cdots \xi_d^{k_d}$. Hence, we have
\[\sum_{|k|=n} A_k \frac{(-i\xi)^k}{|\xi|^{2n}} \hat{L}(a^{-1}\xi) \widehat{\partial^k f}(\xi) = L(a^{-1}\xi)\hat{f}(\xi)\sum_{|k|=n} A_k \frac{(-i\xi)^k(i\xi)^k}{|\xi|^{2n}} = \hat{f}(\xi). \]
For $k \in \NN^d$ with $|k| = n$ let us define 
\[L^k(x) = A_k  \int_{\RR^d}(-i\xi)^k |\xi|^{-2n}\hat{L}(\xi) e^{i\xi\cdot x} \dd x\]
and we have
$f = a^{-n}\sum_{|k|=n} L^k_a * \partial^k f$.
One can use the weighted Young inequality to obtain that
\[\|f\|_{L^p(w)}\lesssim a^{-|n|} \sum_{|k|=n} \|L^k_a\|_{L^1(\llrr^{\lambda'})} \|\partial^k f\|_{L^p(w)}.\]
Furthermore, since $\la a^{-1}x\ra^{\lambda'} \leq \la x\ra^{\lambda'}$, it follows that  
$\|L^k_a\|_{L^1(\llrr^{\lambda'})} \leq \|L^k \|_{L^1(\llrr^{\lambda'})}$.
Finally, we get
\[\|f\|_{L^p(w)}\lesssim a^{-|n|} \|D^n f\|_{L^p(w)},\]
which proves our claim.
\end{proof}

\subsection{Weighted Besov Spaces and standards estimates}\label{sub:weighted_besov}

\begin{defn}
Let $\cA=\{\lambda \in \RR^d\, : \, \frac{3}{4}\le|\lambda|\leq \frac{8}{3} \}$. There exists two radial function $\chi$ and $\varphi$ such that $\mathrm{supp}(\chi) = B(0,\frac34)$, $\mathrm{supp}(\varphi) \subset \cA$, 
\[\forall \lambda \in \RR^d\, \chi(\lambda) + \sum_{j\ge 0} \varphi(2^{-j}\lambda) = 1,\]
and for $j\ge 1$,
\[\mathrm{supp}(\chi)\cap \mathrm{supp}(\varphi(2^{-j}\cdot)) = \emptyset\]
and for $|j-j'|\ge 2$,
\[\mathrm{supp}(\varphi(2^{-j}\cdot)) \cap \mathrm{supp}(\varphi(2^{-j'}\cdot)) = \emptyset.\]
For all $f\in \cS'$ and $j\ge 0$, we define the in-homogeneous Paley-Littlewood blocks by
\[\Delta_{-1}f = \cF^{-1}\big(\chi \hat{f}\big),\quad \mathrm{and} \quad \Delta_j f = \cF^{-1}( \varphi(2^{-j}\cdot) \hat{f} ),\]
where 
$\hat{f}$ denotes the Fourier transform of $f$ and $\cF^{-1}$ the inverse Fourier transform.
\end{defn}

Note that the Paley-Litllewood blocks define a nice appproximation of the unity. We refer to \cite{BahCheDan} proposition 2.12 for a proof. 

\begin{prop}
For all $f\in \cS'$, let us define for all $j\ge -1$ $\cS_j f = \sum_{j' \leq j-1} \Delta_j f$. Then 
\[f = \lim_{j\to \infty} \cS_j f \quad \text{in} \quad \cS'.\]
\end{prop}

\begin{defn}
Let $1\leq p,q \leq +\infty$, and let $\kappa \in \RR$ and $w$ be an admissible weight. For a distribution $f\in \cS'(\RR^d;\RR)$ we define the (in-homogeneous) weighted Besov norm by
\[\|f\|_{B^\kappa_{p,q}(w)} = \Big\| \big(2^{\kappa j} \|\Delta_j f \|_{L^p(\RR^d,w)}\big)_{j\ge -1}\Big\|_{\ell^q(\NN\cup\{-1\})}\]

where $\|f\|_{L^p(\RR^d,w)} = \left(\int_{\RR^d} f(x)^p w(x)^p\dd x\right)^{\frac{1}{p}}$.
When $w\equiv 1$, we only write $B^\kappa_{p,q}$.
\end{defn}

We gather here some basic properties of (weighted) Besov spaces. 

\begin{prop}\label{prop:equivalent_norms}
Let $w$ be an admissible weight. 
\begin{itemize}
\setlength\itemsep{0.5em}
\item[{\rm (i)}]
The space
\[B^s_{p,q}(w) = \{f\in\cS'(\RR^d;\RR)\, : \, \|f\|_{B^\kappa_{p,q}(w)} < +\infty\}\]
does not depend on the choice of $\varphi$ and $\chi$.
\item[{\rm (ii)}] The two following quantities
$\|f\|_{B^\kappa_{p,q}(w)}$  and $ \|wf\|_{B^\kappa_{p,q}}$
are equivalent norms on $B^{\kappa}_{p,q}(w)$.
\item[{\rm (iii)}]
For all $n\ge 0$,
$\|D^n f\|_{B^\kappa_{p,q}(w)} \lesssim \|f\|_{B^{\kappa+n}_{p,q}(w)}$.
\item[{\rm (iv)}]
Let $1\leq p \leq p' \leq +\infty$ and $1\leq q \leq q' \leq +\infty$, then for all $\eps>0$,
\[\|f\|_{B^{\kappa-d\left(\frac{1}{p}-\frac{1}{p'}\right)}_{p',q}(w)} \lesssim \|f\|_{B^\kappa_{p,q}(w)} \lesssim \|f\|_{B^\kappa_{p',q}\left(\langle\cdot\rangle^{d\left(\frac{1}{p}-\frac{1}{p'}\right)+\eps}w\right)}\]
and
\[\|f\|_{B^{\kappa}_{p,\infty}(w)} \lesssim \|f\|_{B^{\kappa}_{p,q}(w)} \lesssim \|f\|_{B^{\kappa-\eps}_{p,q'}(w)}.\]
\item[{\rm (v)}]
For all $\eps,\delta>0$ and all $\kappa \in \RR$ and all $1\leq p,q \leq +\infty$,
\[B^\kappa_{p,q}(w)\quad \text{is compactly embedded in} \quad B^{\kappa-\eps}_{p,q}(\llrr^{-\delta} w)\]
\item[{\rm (vi)}]
Suppose that $\kappa>0$ and $\kappa \notin \NN$ and for $f: \RR^d \mapsto \RR$ let us define 
\[\|f\|_{\cC^\kappa(w)} = \sum_{|k|\leq [\kappa]} \sup_{x \in \RR^d} |w(x)\partial^k f (x)| + \sum_{|k| = [\kappa] } \sup_{0< |h|\leq 1}\sup_{x \in \RR^d} \frac{w(x)|\partial^k f(x+h) - \partial^k f(x)|}{|h|^{\kappa-[\kappa]}}.\]
Then 
\[\cC^\kappa(w) = \{f\, : \, \|f\|_{\cC^\kappa(w)}<+\infty\} = B^\kappa_{\infty,\infty}(w)\]
and furthermore $\|\cdot\|_{\cC^\kappa(w)}$ and $\|\cdot\|_{B^{\kappa}_{\infty,\infty}(w)}$ are equivalent norms on this space. 
\end{itemize}
\end{prop}

\begin{proof}
We only prove the weighted inequality in the fourth point, and we refer \cite{Triebel2006} and the references therein for the other ones. The first and the third inequalities are  direct consequences of Lemma \ref{lem:bernstein}. For the second one, let us take $1\leq p< p' <+\infty$. We have, thanks to Jensen inequality, for any $\eps>0$
\begin{align*}
    \|\Delta_j f\|_{L^p(w)}^{p'} 
        \lesssim_{d,\eps} &
    \int_{\RR^d} |\Delta_j f(x)|^{p'} w(x)^{p'} \la x \ra^{\frac{(d+\eps)p'}{p}}\la x \ra^{-\frac{(d+\eps)p'}{p'}}\dd x\\ 
    =&
     \int_{\RR^d} |\Delta_j f(x)|^{p'} \left(w(x) \la x \ra^{(d+\eps)\left(\frac{1}{p} - \frac{1}{p'}\right)}\right)^{p'}\dd x .
    \end{align*}
    The constant in the previous inequality does not depends on $p'$. This gives
    \[\|\Delta_j f\|_{L^p(w)} \lesssim_{d,\eps}\|\Delta_j f\|_{L^{p'}\left(\llrr^{(d+\eps)\left(\frac{1}{p}-\frac{1}{p'}\right)} w\right)}.\]
\end{proof}

Finally, in order to deal with product of elements in Besov space, we give the following result, which can be proved thanks to standard techniques (see \cite{BahCheDan} Lemma 2.69 and 2.84 and \cite{Mourrat2017Global} Theorem 3.17 and Corollary 3.19 and 3.21). Note that it mostly relies on H\"older and Young inequality, and therefore is available in the context of weighted spaces. 

\begin{prop}\label{prop:bony_estimates}[Corollary 2.86 in \cite{BahCheDan} and Corollary 3.19 in \cite{Mourrat2017Global}] \label{prop:bony}
Let $w$ be an admissible weight, $\kappa_2 \leq \kappa_1 $ with $\kappa_1 \ge 0$ and suppose that $ \kappa_1 + \kappa_2 > 0$, and let $1\leq p, p_1,p_2,q \leq +\infty$ 
such that $\frac{1}{p}=\frac{1}{p_1}+\frac{1}{p_2}$. Let $\eps,\delta >0$. Then for all $f\in \cB^{\kappa_1}_{p_1,q}(w)$ and all $g\in \cB^{\kappa_2}_{p_2,q}(w)$,
\begin{itemize}
\setlength\itemsep{0.5em}
    \item[{\rm (i)}] $\left(\cS_j f \cS_j g\right)_{j\ge 0}$ converges in $B^{\kappa_2 - \eps}_{p,q}(\llrr^{-\delta} w)$ to a limit in $B^{\kappa_2}_{p,q}(w)$.
    \item[{\rm (ii)}]
    We have
    \[\left\|\lim_{j} \cS_j f \cS_j g\right\|_{B^{\kappa_2}_{p,q}(w)} \lesssim \left\|f \right\|_{B^{\kappa_1}_{p,q}(w)} \left\|  g\right\|_{B^{\kappa_2}_{p,q}(w)}\]
    \item[{\rm (iii)}]
    When $\kappa_2 \ge 0$,  $\lim_{j} \cS_j f \cS_j g= fg$ the standard point wise product.
\end{itemize}
\end{prop}
 
 \begin{rem}\label{def:product}
 The previous proposition shows that the limit does not depend on the choice of the blocks, and therefore it extends canonically the notion of product of functions to a product of distributions, as soon as $\kappa_1 + \kappa_2 >0$, with $\kappa_1\geq 0$. For this reason, we will denote by $fg = \lim_{j} \cS_j f \cS_j g$, and this is a bi-linear functional from $B^{\kappa_1}_{p_1,q}(w) \times B^{\kappa_2}_{p_2,q}(w)$ to $B^{\kappa_2}_{p,q}(w)$.
 \end{rem}

\subsection{Heat kernel estimates}\label{sub:heat_kernel}

In order to deal with heat kernel estimates on weighted Besov spaces, we first need some heat kernel estimates for function whose Fourier transform has a support in an annulus. In order to deal with non-Gaussian noise, we need to consider heat semi-group for fractional Laplacian. For a full study of the fractional Laplacian, we refer to \cite{kwasnickiTenEquivalentDefinitions2017}.  Here we just define the fractional Laplacian for smooth functions.

\begin{defn}
Let $\alpha \in (0,2]$. For any function $f \in \cS(\RR^d;\RR)$ we define the fractional Laplace operator $\Delta^{\frac{\alpha}{2}} = -(-\Delta)^{\frac{\alpha}{2}}$ by
\[\Delta^{\frac{\alpha}{2}} f = \cF^{-1}(-|\cdot|^\alpha \hat{f}(\cdot)).\]
Furthermore, we define the semi-group associated to $\Delta^{\frac{\alpha}{2}}$ \[P^{\frac{\alpha}{2}}_t f = \cF^{-1}\left(e^{-|\cdot|^\alpha t} \hat{f}(\cdot)\right).\]
Finally we extend this definition to the whole space $\cS'$ by the standard procedure.
\end{defn}

Note that when $\alpha = 2$, the previous definition gives the standard Laplace operator. When this is the case we simply write $P$ instead of $P^1$.

\begin{prop}\label{prop:heat_kernel_estimates}
Let $w$ be an admissible weight and $\lambda'$ be defined as in \eqref{eq:lambda'}. Let $\alpha \in (0,2]$.  Let $\cA=\{\xi \in \RR^d\, :\, c_1\leq |\xi| \leq c_2\}$ be an annulus, let $a\ge 1$ and let $f :\RR^d \to \RR^d$ be a function such that $supp \hat f  \subset a\cA$. Let $0<s\leq t$. There exists a constant $c>0$ such that for all $p\ge 1$, 
\[\|P^{\frac{\alpha}{2}}_t f\|_{L^p(w)} \lesssim e^{-c t a^\alpha } \|f\|_{L^p(w)},\]
and for all $\rho\ge 0$,
\[\|(P_t - P_s )f\|_{L^p(w)} \lesssim a^{-\alpha\rho}\left|\frac{1}{s^\rho}-\frac{1}{t^\rho}\right| \|f\|_{L^p(w)}.\]
Finally, for all $s \leq u \leq \tau' \leq \tau$, and for all $\rho>0$,
\begin{multline*}
    \Big\|\big((P_{\tau-s}-P_{\tau-u})-(P_{\tau'-s}-P_{\tau'-u})\big)f\Big\|_{L^p(w)}\\ 
        \lesssim 
    a^{-\alpha\rho}\left(\frac{1}{(\tau'-u)^\rho} - \frac{1}{(\tau' - s)^\rho} - \frac{1}{(\tau-u)^\rho} + \frac{1}{(\tau - s)^\rho}\right) \|f\|_{L^p(w)}
    \end{multline*}
\end{prop}

\begin{proof}
We follow the proof of \cite{BahCheDan}, Lemma 2.4. The first affirmation is the result of this lemma but in the context of weighted spaces and of fractional operator. We also refer to \cite{Mourrat2017Global} Lemma 2.10 to a proof. Since the proof of the second and third points are similar to the proof of the first one, we will not detail the first point. For the second and third one, let us define
\[E(s,t;\xi) := e^{-|\xi|^\alpha t} - e^{-|\xi|^\alpha s}.\]
Note that in that case, we have 
\[(P_t - P_s )f(x) = \int_{\RR^d}E(s,t;\xi)\hat{f}(\xi) e^{i\xi x} \dd x.\]

Thanks to the hypothesis, there also exists a smooth function $\varphi : \RR^d \to \RR^d$ such that $\varphi \equiv 1$ on $\cA$ and $\mathrm{supp} \varphi \subset\{\xi \in \RR^d\, :\, \frac{c_1}{2} \leq |\xi| \leq 2 c_2\}$. 
Let us define for all $s\leq t$
\[K(s,t;x) = \int E(s,t;\xi) \varphi(\xi) e^{i\xi x}d \xi,\]
and $K_a(s,t,x) = a^{d} K(s,t,a x)$.
Hence we have 
\begin{align*}
(P_t-P_s)f(x)
=&    
    \int_{\RR^d} E(s,t;\xi)  \hat{f}(\xi) e^{i\xi \cdot x}   d\xi  \\
=& 
    \int_{\RR^d} E(s,t;\xi) \varphi(a^{-1} \xi)  \hat{f}(\xi) e^{i\xi \cdot x}   d\xi\\
=&
    a^d \int_{\RR^d} E(s,t;a\xi) \varphi(\xi)  \hat{f}(\xi) e^{i\xi \cdot a x}   d\xi.
\end{align*}

And since $E(s,t,a \xi) = E(a^\alpha s, a^\alpha t; \xi) $, one has 
\[(P_t - P_s)*f(x) = K_a(a^\alpha s, a^\alpha t;\cdot)*f(x).\]
Finally, thanks to Young inequality, one has
\[\|(P_t-P_s)*f \|_{L^p(w)} \leq \|K_a(a^\alpha s,a^\alpha t;\cdot)\|_{L^1(\llrr^{\lambda'})} \|f\|_{L^p(w)}.\]

Note also that since $\lambda' \ge 0 $ and $a\ge 1$, we have $\la a^{-1} x \ra^{\lambda'} \leq \la x \ra^{\lambda'} $, and 
\begin{align*}
    \|K_a(a^\alpha s,a^\alpha t;\cdot)\|_{L^1(\llrr^{\lambda'})} 
    =&
    a^{d}\int_{\RR^d} K_{a^\alpha s, a^\alpha t, a x} \la x \ra^{\lambda'} \dd x \\
    =& 
    \int_{\RR^d} K_{a^\alpha s, a^\alpha t, a x} \la a^{-1} x\ra^{\lambda'} \dd x \\
    =& 
    \leq  a^{d}\int_{\RR^d} K_{a^\alpha s, a^\alpha t, x} \la x \ra^{\lambda'} \dd x \\
    =& 
    \|K(a^\alpha s,a^\alpha t;\cdot)\|_{L^1(\llrr^{\lambda'})}
\end{align*}
Hence, it is enough to prove the proposition for $a = 1$ and for all $0<s'\leq t'$, and then specify $s'=a^\alpha s$ and $t'=a^\alpha t$.

Let $M\in \NN$ such that $2M>d+\lambda'$. In order to prove the proposition, it is then enough to bound $\big|(1+|x|^2)^{M} K(s,t;x) \big|$ by the wanted quantity $\frac{1}{s^\rho}-\frac{1}{t^\rho}$. We have
\begin{align*}
    \left(1+|x|^2\right)^M K(s,t;x) 
=& 
    \int_{\RR^d} \left( (1-\Delta)^M e^{ix\cdot}\right)(\xi) \varphi(\xi) E(s,t;\xi) d\xi\\
=&
    \int_{\RR^d}   e^{ix\xi} (1-\Delta)^M \big(\varphi(\cdot) E(s,t;\cdot)\big)(\xi) d\xi
\end{align*}
And thanks to the Fa\`a di Bruno formula, there exists constants $(c_{\nu,\kappa})$ where $\nu$ and $\kappa$ are multi-indices such that 
\[\left(1+|x|^2\right)^M K(s,t;x) = \sum_{|\nu| + |\kappa| \leq 2M} c_{\nu,\kappa} \int_{\RR^d} e^{i \xi x}  \partial^{\nu}\varphi(\xi) \partial^{\kappa} E(s,t;\xi) d\xi.\]
Hence, in order to prove the proposition, one only has to bound all the derivatives up to order $2M$ of $E(s,t;\cdot)$ for $\xi \in \cA$ by the wanted quantity $\frac{1}{s^{\rho}}-\frac{1}{t^\rho}$.

Note that the same strategy could be used in the case where we have the rectangular increment of the semi-group if we replace $E(s,t;\xi)$ by
\[F(s,u,\tau',\tau;\xi) := \left( e^{-|\xi|^\alpha (\tau-s)} - e^{-|\xi|^\alpha (\tau-u)} \right) - \left( e^{-|\xi|^\alpha (\tau'-s)} - e^{-|\xi|^\alpha (\tau'-u)}\right).\]
The same partial conclusion holds, one only has to control all the derivatives of $F(s,u,\tau',\tau;\cdot)$ up to order $2M$ for $\xi \in \cA$.

Furthermore, observe  that 
\[E(s,t;\xi) = \int_s^t -|\xi|^\alpha e^{-r |\xi|^\alpha} \dd r.\]
Hence, by a direct induction and since $\xi \in \cA$, for every multiindexe $k=(k_1\cdots,k_d)\in \NN^d$, there exists a polynomial  
\[P_{k,\xi}(t) = \sum_{l=0}^{|k|} a_{l}^k(\xi) t^l,\]
where for all $l\in\{0,\cdots,n\}, $ $\xi \to a_l^k(\xi)$ are non-negative smooth functions on $\cA$, and for all $\xi \in \cA$, $a_n^k(\xi) \neq 0$ and 
such that 
\[\partial^k \big(|\cdot|^2 e^{-r |\cdot|^\alpha}\big)(\xi) = P_{k,\xi}(r)e^{-r|\xi|^\alpha}. \]
Furthermore, since $c_1\le|\xi|\leq c_2$, there exists a constant $c>0$ (depending on $\cA$) such that
\[\big|\partial^k \big(|\cdot|^\alpha e^{-r |\cdot|^\alpha}\big)(\xi)\big| \lesssim \tilde{P}_{k,\xi}(r) e^{-r |\xi|^2} \lesssim_{k,\cA} e^{-c r} \lesssim r^{-\rho-1}\]
for any $\rho\ge -1$. Hence
 \[|\partial^k E(s,t;\cdot)| \lesssim \int_s^t r^{-\rho} \dd r \lesssim \frac{1}{s^\rho} - \frac{1}{t^\rho}.\]
 With the previous discussion, this gives the wanted result, when we replace $t$ by $a^\alpha t$ and $s$ by $a^\alpha s$.
In order to deal with the last estimates, one only has to remember that 
\begin{align*}
    F(s,u,\tau',\tau; \xi) 
        = & 
    -\int_u^s |\xi|^\alpha \left( e^{-|\xi|^\alpha (\tau - r)} - e^{-|\xi|^\alpha (\tau' - r)}\right) \dd r \\
    =&
    \int_s^u \int_{\tau'-r}^{\tau - r} |\xi|^{2\alpha} e^{-|\xi|^\alpha v} \dd v \dd r
\end{align*}
The same argument as before gives the bound 
\[|\partial^k F(s,u,\tau',\tau; \cdot)| \lesssim \int_s^u \int_{\tau'-r}^{\tau - r}  e^{ - c v} \dd v \dd r.\]
Finally, for any $\rho \ge 0$,
\begin{align*}
|\partial^k F(s,u,\tau',\tau; \cdot)| 
    \lesssim & 
\int_s^u \int_{\tau'-r}^{\tau - r} v^{-(\rho+2)} \dd v \dd r
\\
\lesssim&
\frac{1}{(\tau'-u)^\rho} - \frac{1}{(\tau' - s)^\rho} - \frac{1}{(\tau-u)^\rho} + \frac{1}{(\tau - s)^\rho}
\end{align*}
Again, this give the wanted result when recalling that one must replace $\tau$,$\tau'$,$u$ and $s$ by $a^\alpha \tau$, $a^\alpha \tau'$, $a^\alpha u$ and $a^\alpha s$.
\end{proof}

Let us gives the following useful and straightforward corollary for the action of the fractional heat semi-group on weighted Besov spaces. Let us first remind a rather elementary be useful lemma (\cite{TindelDeya2009}, Lemma 4.4). For the sake of the reader, we provide a full proof of it. 

\begin{lem}\label{lem: 44}
Let $\rho\ge 0$ and $\theta\in[0,1]$. There is a constant $c>0$ such that for any $0 < s \leq t$,
\[\frac{1}{s^\rho} - \frac{1}{t^\rho} \leq c (t-s)^\theta s^{-(\rho + \theta)}.\]
Let $\rho\ge 0$ and $\theta,\theta' \in[0,1]$. There exists a constant such that
\[\frac{1}{(\tau'-u)^\rho} - \frac{1}{(\tau'-s)^\rho} - \frac{1}{(\tau-u)^\rho} + \frac{1}{(\tau-s)^\rho} \leq c (\tau-\tau')^\theta (u-s)^{\theta'} (\tau'-u)^{-(\rho + \theta + \theta')}\]
\end{lem}

\begin{proof}
Let us remark that 
\[\frac{1}{s^\rho} - \frac{1}{t^\rho} = (1+\rho)\int_s^t r^{-(\rho+1)}\dd r \leq (1+\rho)(t-s) s^{-(\rho+1)}.\]
The result follows by a standard interpolation between the two inequalities. 
For the second inequality, set 
\[B =\frac{1}{(\tau'-u)^\rho} - \frac{1}{(\tau'-s)^\rho} - \frac{1}{(\tau-u)^\rho} + \frac{1}{(\tau-s)^\rho}.\] We have, thanks to the same integral representation, 
\[B\lesssim (\tau'-u)^{-\rho},\]
\[B \lesssim (\tau - \tau') (\tau'-u)^{-(\rho+1)},\]
\[B \lesssim (u-s)(\tau'-u)^{-(\rho+1)},\]
and
\[B \lesssim (u-s)(\tau-\tau') (\tau'-u)^{-(\rho+2)}.\]
Let us suppose, without loss of generality that $\theta' \leq \theta$ that is $\theta' = \alpha \theta$ for some $0\leq \alpha<1$. We have, using the first and the last inequalities,
\[B \lesssim (\tau-\tau')^\theta (u-s)^{\theta} (\tau'-u)^{-(\rho + 2\theta)}.\]
Using the first and the second one, we have
\[B \lesssim (\tau-\tau')^\theta (\tau'-u)^{-(\rho + 2\theta)}.\]
We interpolate those last two inequalities to have
\[B \lesssim (\tau-\tau')^\theta (u-s)^{\alpha \theta} (\tau'-u)^{-(\rho + 2\alpha\theta + (1-\alpha)\theta)} = (\tau-\tau')^\theta (u-s)^{ \theta'} (\tau'-u)^{-(\rho + \theta + \theta')}.\]
\end{proof}

\begin{cor}\label{cor: heat kernel estimates}
Let $w$ be an admissible weight and $\alpha \in (0,2]$. Let $\kappa \in \RR$ and let $1 \leq p,q \leq +\infty$. Let $0\leq t$, then for all $\rho\ge 0$,
\[\|P^{\frac{\alpha}{2}}_t f\|_{B^{\kappa+\alpha\rho}_{p,q}(w)} \lesssim t^{-\rho}  \|f\|_{B^{\kappa}_{p,q}(w)}.\]
For all $\theta\in [0,1]$ and all $\rho>0$,
\[\|(P^{\frac{\alpha}{2}}_t-P^{\frac{\alpha}{2}}_s)f\|_{B^{\kappa+\alpha\rho}_{p,q}(w)} \lesssim |
(t-s)^\theta s^{-(\rho+\theta)}\|f\|_{B^{\kappa}_{p,q}(w)}.\]
Furthermore, for any $\theta,\theta'\in[0,1]$ and all $\rho\ge 0$,
$s \leq u < \tau' \leq \tau$,
\[\left\|\big((P^{\frac{\alpha}{2}}_{\tau - s}-P^{\frac{\alpha}{2}}_{\tau - u}) - (P^{\frac{\alpha}{2}}_{\tau' - s}-P^{\frac{\alpha}{2}}_{\tau' - u})\big)f\right\|_{B^{\kappa+\alpha\rho}_{p,q}(w)} \lesssim (\tau-\tau')^\theta (u-s)^{\theta'} (\tau'-u)^{-(\rho + \theta + \theta')}\|f\|_{B^{\kappa}_{p,q}(w)}.\]
\end{cor}

\begin{proof}
The first bound is a direct consequence of the first bound of Proposition \ref{prop:heat_kernel_estimates}. Indeed, for all $j\ge 0$,
\[\|\Delta_j P^{\frac{\alpha}{2}}_t f\|_{L^P(w)} = \| P_t (\Delta_j f)\|_{L^P(w)} \lesssim e^{-c2^{\alpha j} t} \|\Delta_j f\|_{L^P(w)} \lesssim 2^{-\alpha\rho j} t^{-\rho} \|\Delta_j f\|_{L^P(w)},\]
and the result follows via the definition of weighted Besov spaces. 
For the second bound, let us remark that we have for all $0 < s \leq t$, and thanks to Proposition \ref{prop:heat_kernel_estimates}, 
\[\|\Delta_j (P^{\frac{\alpha}{2}}_t - P^{\frac{\alpha}{2}}_s) f\|_{L^P(w)} = \lesssim \left(\frac{1}{s^\rho} - \frac{1}{s^\rho}\right) 2^{-\alpha\rho j} \|\Delta_j f\|_{L^P(w)} \lesssim 2^{-\alpha\rho j}(t-s)^\theta s^{-(\rho+\theta)} \|\Delta_j f\|_{L^p(w)}.\]
where the last inequality comes from the previous lemma. This allows us to derive the result by using the definition of $B$ spaces. 
Finally, note that we have thanks to Proposition \ref{prop:heat_kernel_estimates} and Lemma~\ref{lem: 44},
\[
    \Big\|\Delta_j\big((P^{\frac{\alpha}{2}}_{\tau-u}-P^{\frac{\alpha}{2}}_{\tau-s})-(P^{\frac{\alpha}{2}}_{\tau'-u}-P^{\frac{\alpha}{2}}_{\tau'-s})\big)f\Big\|_{L^p(w)} 
    \lesssim 2^{-\alpha\rho j}(\tau-\tau')^\theta (u-s)^{\theta'}  (\tau' - u)^{\rho+\theta + \theta'} \|\Delta_j f\|_{L^p(w)}.  
\]
And again one can conclude with the definition of the weighted Besov spaces.
\end{proof}

\section{Cauchy-Lipschitz theorem for mSHE in standard case.}\label{sec:CL-standard}

We give a short proof of local well-posedness of the mSHE in a simple context. For more on this, one can consult \cite{liuStochasticPartialDifferential2015} and the generalization for more irregular noise with more involve techniques \cite{Gubinelli2016Fourier} and \cite{Hairer2013Regularity}. Note that 

\begin{thm}
Let $\alpha \in (0,2]$. Let $\vartheta \in(0,\min(1,\alpha))$, let $\vartheta<\beta<\alpha-\vartheta$. Let $g\in \cC^2$ and let $u_0 \in \cC^\beta$ and $\xi \in \cC^{-\vartheta}$
There exists a unique local solution of the equation 
\[\partial_t u = \Delta^{\frac{\alpha}{2}} u + g(u) \xi\]
in the mild form
\[u(t,\cdot) = P^{\frac{\alpha}{2}}_tu_0 + \int_s^t P^{\frac{\alpha}{2}}_{t-s}\xi g(u(s,\cdot))\dd s.\]
\end{thm}

\begin{proof}
First, let us take $u\in \cC^\beta$. We have
\[\|g(u(t,\cdot)\|_{\cC^{\beta}} \leq \|g\|_{\cC^{2}} \|u(t,.)\|_{\cC^\beta}. \]
For $u\in C([0,T];\cC^{\beta})$, let us remark that thanks to standard Bony estimates in Besov-H\"older spaces (Proposition \ref{prop:bony})
\[\|\xi g(u(t,\cdot)\|_{\cC^{-\vartheta}} \lesssim  \|g\|_{\cC^{2}} \|\xi\|_{\cC^{-\vartheta}} \|u(t,.)\|_{\cC^\beta}.\]
Finally for $0\leq s < t \le T $,
\[\|P_{t-s}\xi g(u(t,\cdot))\|_{\cC^{\beta}} \lesssim  \frac{1}{(t-s)^{\frac{\beta + \rho}{\alpha}}}\|g\|_{\cC^{2}} \|\xi\|_{\cC^{-\vartheta}} \|u(t,.)\|_{\cC^\beta}.\]
Hence, the application $\Gamma$
\[\Gamma(u)(t,x) = P_t u_0 + \int_0^t P_{t-s}\xi g(u(s,\cdot)) \dd s\]
is well-defined from $C^{0}([0,T];\cC^{\beta})$ to itself. Furthermore, 
Let us remark that for $u,v \in \cC^\beta$,
\[\|g(u) - g(v) \|_{\cC^{\beta}} \lesssim \|g\|_{\cC^2}\|u-v\|_{\cC^{\beta}},\]
hence for $u,v\in C^0([0,T];\cC^{\beta})$,
\begin{multline*}
\sup_{t\in[0,T]}\|\Gamma(u)(t,\cdot)-\Gamma(v)(t,\cdot) \|_{\beta}
\lesssim  \int_0^t \frac{1}{(t-s)^{\frac{\beta + \rho}{\alpha}}} \|g\|_{\cC^{2}} \|\xi\|_{\cC^{-\vartheta}} \|u(t,.) - v(t,.)\|_{\cC^\beta}  \dd s \\
\lesssim T^{1-\frac{\beta + \rho}{\alpha}} \sup_{t\in[0,T]}\|u(t,.) - v(t,.)\|_{\cC^\beta}.
\end{multline*}
Hence, be standard Schauder fixed point, for $T$ small enough, there is a unique $u\in C^{0}([0,T];\cC^\beta)$.
\end{proof}

\bibliographystyle{plain}

\end{document}